\documentclass{amsart}

\usepackage{amsmath,amssymb,amsthm,bbm} 
\usepackage{hyperref} 
\usepackage{cleveref} 
\usepackage{geometry} 
\usepackage{xcolor} 
\usepackage{tikz} 
\usepackage{enumerate}
\usepackage{csquotes}

\definecolor{theme}{RGB}{15, 87, 24} 
\definecolor{lighttheme}{RGB}{51, 153, 63} 
\definecolor{block}{RGB}{102, 60, 0} 
\definecolor{alert}{RGB}{212, 43, 43} 

\def\pagespace{3cm}
\hypersetup{
    colorlinks=true,
    linkcolor=theme,
    citecolor=lighttheme
}

\newtheorem{thm}{Theorem}[section]
\newtheorem{prop}[thm]{Proposition}
\newtheorem{lemma}[thm]{Lemma}
\newtheorem{cor}[thm]{Corollary}
\newtheorem{conj}{Conjecture}

\theoremstyle{remark}
\newtheorem{remark}{Remark}

\newcommand{\bN}{\mathbb{N}}
\newcommand{\bR}{\mathbb{R}}
\newcommand{\prob}{\mathbb{P}}
\newcommand{\expec}{\mathbb{E}}
\renewcommand{\O}{\mathcal{O}}
\renewcommand{\o}{o}
\newcommand{\eps}{\varepsilon}
\DeclareMathOperator{\Leb}{Leb}
\DeclareMathOperator{\dist}{dist}
\DeclareMathOperator{\rec}{rec}
\newcommand{\lChild}{0}
\newcommand{\rChild}{1}
\newcommand{\Troot}{\varnothing}
\newcommand{\lT}{\mathcal{T}}
\newcommand{\rV}{\mathbb{V}}
\newcommand{\pts}{\mathcal{P}}
\newcommand{\grubeling}{\overset{\mathrm{ssc}}{\longrightarrow}}
\newcommand{\lbst}[1]{\lT\langle#1\rangle}
\newcommand{\sigmaP}[1]{\sigma\langle#1\rangle}
\newcommand{\bst}[1]{T\langle#1\rangle}
\newcommand{\lPT}[2]{\lT^{#1}_{#2}}
\newcommand{\poisson}[1]{\textsc{Poisson}\left(#1\right)}
\newcommand{\binomial}[2]{\textsc{Binomial}\left(#1,\,#2\right)}
\newcommand{\lTleft}{\lT_{\mathrm{left}}}
\newcommand{\lTright}{\lT_{\mathrm{right}}}
\newcommand{\lTnleft}{\lT^n_{\mathrm{left}}}
\newcommand{\lTnright}{\lT^n_{\mathrm{right}}}
\newcommand{\lTinfleft}{\lT^{\infty}_{\mathrm{left}}}
\newcommand{\lTinfright}{\lT^{\infty}_{\mathrm{right}}}
\newcommand{\Tlim}[1]{\psi_{#1}}
\newcommand{\assumption}[1]{$\mathrm{(\hyperlink{A#1}{A#1})}$}
\newcommand{\itemlink}[1]{$\mathrm{(\hyperlink{item#1}{#1})}$}
\newcommand\One{\mathbf{1}}
\newcommand{\height}[1]{h\left(\lbst{#1}\right)}
\newcommand{\LIS}[1]{\mathrm{LIS}\left(#1\right)}
\newcommand{\LDS}[1]{\mathrm{LDS}\left(#1\right)}

\title{Binary search trees of permuton samples}
\author[B. Corsini]{Beno\^it Corsini}
\address[BC]{Department of Mathematics and Computer Science, Eindhoven University of Technology, 5600 MB Eindhoven, The Netherlands}
\email{benoitcorsini@gmail.com}

\author[V. Dubach]{Victor Dubach}
  \address[VD,VF]{Université de Lorraine, CNRS, IECL, F-54000 Nancy, France}
\email{victor.dubach@univ-lorraine.fr, valentin.feray@univ-lorraine.fr}

 \author[V. Féray]{Valentin Féray}
\date{}

\begin{document}

\maketitle

\begin{abstract}
    Binary search trees (BST) are a popular type of data structure when dealing with ordered data.
    Indeed, they enable one to access and modify data efficiently, with their height corresponding to the worst retrieval time.
    From a probabilistic point of view, binary search trees associated with data arriving in a uniform random order are well understood, but less is known when the input is a non-uniform random permutation.

    We consider here the case where the input comes from i.i.d.~random points in the plane with law $\mu$,
    a model which we refer to as a {\em permuton sample}.
    Our results show that the asymptotic proportion of nodes in each subtree depends on the behavior of the measure $\mu$ at its left boundary, while the height of the BST has a universal asymptotic behavior for a large family of measures $\mu$.
    Our approach involves a mix of combinatorial and probabilistic tools, namely combinatorial properties of binary search trees, coupling arguments, and deviation estimates.
\end{abstract}

\section{Introduction}

\subsection{Context and informal description of our results}
\label{ssec:context}
A {\em binary search tree} (BST) is a rooted binary tree where nodes carry labels (which are real numbers) and where, for each vertex $v$, all labels of vertices in the left-subtree (resp.~right-subtree) attached to $v$ are smaller (resp.~bigger) than the label of $v$.
Binary search trees are a popular type of data structure for storing ordered data.
One key feature is that the {worst-case} complexity of basic operations (lookup, addition or removal of data) is proportional to the height of the tree.

Given a BST $\mathcal T$ and a real number $x$ distinct from the labels of $\mathcal T$, there is a unique way to insert $x$ into $\mathcal T$, i.e.~there is a unique BST $\mathcal T^{+x}$ obtained from $\mathcal T$ by adding a new node with label $x$.
Iterating this operation starting from the empty tree and a sequence $y=(y_1,\ldots,y_n)$ of distinct values, we get a BST $\mathcal T\langle y\rangle$ with $n$ nodes.
An example of the sequence of trees obtained from $y=(2,4,1,6,3,5)$ can be found in Figure~\ref{fig:BST}.
The shape of $\mathcal T\langle y\rangle$ (i.e.~the underlying binary tree without node labels) depends only on the relative order of the numbers $y_1,\dots,y_n$, and not on their actual value.
We can thus assume without loss of generality that the sequence $y$ is a permutation $\sigma$ of the integers from 1 to $n$, and write $\mathcal T\langle \sigma\rangle = \mathcal T\langle \sigma_1,\dots,\sigma_n \rangle$ in this case.

\begin{figure}[htb]
    \[\includegraphics[scale=.8]{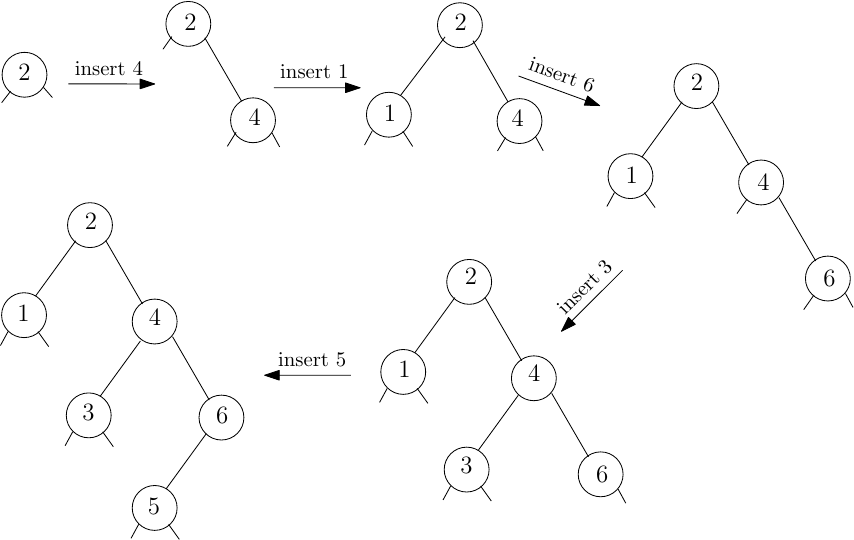}\]
    \caption{Iterative construction of the BST associated with the sequence $y=(2,4,1,6,3,5)$. 
    Let us detail the step where 3 is inserted. Since 3 is bigger than the root label (here 2), it should be added in the right-subtree attached to the root.
    We then compare 3 to the label of the root of that subtree, which is 4 in our example. 
    Since 3 is smaller than 4, it should be added in the left subtree attached to $4$.
    This subtree is empty at this stage, so we simply attach 3 to the left of 4.}
    \label{fig:BST}
\end{figure}

In the worst case, the tree $\lbst{\sigma}$ has height $n-1$ and further operations such as lookup, addition or removal of data will have a linear complexity, which is far from optimal.
However it has been proven by Devroye~\cite{devroye1986note} that, if $\sigma$ is uniformly distributed in the symmetric group $S_n$, then the height $\height{\sigma}$ is asymptotically equivalent to $c^* \log(n)$ for some constant $c^*$, yielding a much better complexity for later operations.
Assuming that $\sigma$ is uniformly distributed means that the data used to construct our BST arrived in a completely random order, which is in general unrealistic.
It seems therefore natural to study BSTs associated with non-uniform random permutations, and in particular to see how Devroye's result is modified when changing the distribution of $\sigma$.

A first step in this direction has been performed in the papers~\cite{addario2021bst-mallows,corsini2023bst-record}, where the BSTs associated with random Mallows and record-biased permutations are studied, showing interesting phase transition phenomena.
In the current paper, we will consider some geometric models of random permutations, sampled via i.i.d.~random points in the plane with some common distribution $\mu$.
These models will be referred to here as {\em permuton samples}, and denoted by $\sigma_\mu^n$; 
they appear naturally in a recently developed theory of limiting objects for large permutations, called permutons \cite{hoppen2013permutons}.
{The goal of studying such models is twofold.
First, it is a much larger but still tractable family of models than those considered before
(permuton samples are indexed by probability measures on the square,
while Mallows and record-biased permutations are one-parameter families of models).
Second, since permutons describe the \enquote{large-scale shape} of permutations, it enlightens the connection between this \enquote{large-scale shape} and the associated BST.}

Our first result (Theorem~\ref{thm:height}) shows that for a large family of permuton samples, the asymptotic behavior of the BST height is the same as the one found by Devroye for uniform random permutations, namely $\height{\sigma_\mu^n}$ is asymptotically equivalent to $c^* \log(n)$.
We also consider the {repartition of nodes in various branches of the BST}, using the formalism of subtree size convergence recently introduced by Gr\"ubel in \cite{grubel2023note}.
In this setting, {Theorem~\ref{thm:limit} below proves the convergence of the BST associated with permuton samples, under some mild assumption}, where the limit object depends on the permuton only through its \enquote{derivative} at the left edge of the unit square $[0,1]^2$ 
(the {\em derivative} of a permuton at $\{0\} \times [0,1]$ does not make sense in general, but {the mild assumption in the theorem precisely postulates its existence}).

In the remaining part of the introduction we present the model of permuton samples (\Cref{ssec:permuton-sample}), state our results precisely (\Cref{ssec:result-height,ssec:result-stscv}) and give an overview of the proofs (\Cref{ssec:proof-overview}).

\subsection{Our model: binary search trees of permuton samples}
\label{ssec:permuton-sample}

There is a natural way to map a (generic) finite set of points $\pts \subset \bR^2$ to a permutation $\sigmaP{\pts}$ and a binary search tree $\lbst{\pts}$, which we describe now.
Let $\pts=\{(x_1,y_1),\ldots,(x_n,y_n)\}$ be a set of points in  $\bR^2$ with distinct $x$- and distinct $y$-coordinates, and let $\{(x_{(1)},y_{(1)}),\ldots,(x_{(n)},y_{(n)})\}$ be its reordering such that $x_{(1)}<\ldots<x_{(n)}$.
Then there exists a unique permutation $\sigma=\sigmaP{\pts}$ such that $(y_{(1)},\dots,y_{(n)})$ and $(\sigma_1,\dots,\sigma_n)$ are in the same relative order.
We let $\lbst{\pts}:=\lbst{y_{(1)},\dots,y_{(n)}}$ and note that the trees $\lbst{\pts}=\lbst{y_{(1)},\dots,y_{(n)}}$ and $\lbst{\sigmaP{\pts}}=\lbst{\sigma_1,\dots,\sigma_n}$ have the same shape since the two sequences have the same relative order.
They do, however, have different sets of labels:
$\{y_1,\dots,y_n\}$ for $\lbst{\pts}$ and $\{1,\dots,n\}$ for $\lbst{\sigmaP{\pts}}$.
These constructions are illustrated in \Cref{fig:BST-points}.
The shape of $\lbst{\pts}$ is indeed the same as that of $\lbst{\sigmaP{\pts}}$ (see the last image in \Cref{fig:BST}).

\begin{figure}[htb]
    \[\includegraphics[scale=1]{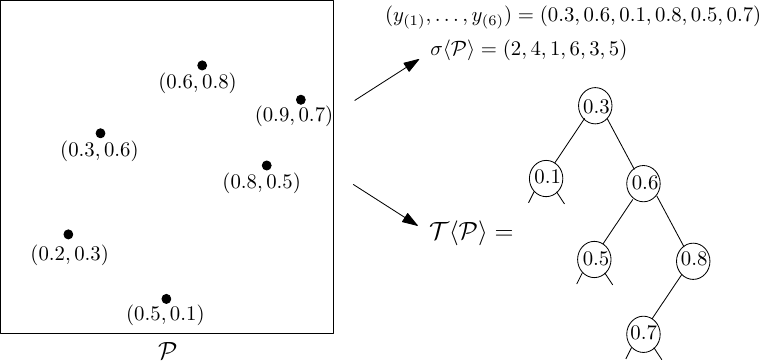}\]
    \caption{A set of points in $\mathbb R^2$ and its associated permutation and binary search tree.}
    \label{fig:BST-points}
\end{figure}

Now consider a probability measure $\mu$ on $\bR^2$ and take a set $\pts^n_\mu$ of $n$ i.i.d.~points in $\bR^2$ with distribution $\mu$.
In order to make sure that the associated permutation and binary search tree are always well-defined, we need the coordinates of the points to be all distinct.
To this extent, we assume for the rest of this work that the projections of $\mu$ on both axes have no atom.
Moreover, since the permutation and the shape of the tree only depend on the relative positions of the points, without loss of generality we can re-scale $\mu$ so that its support is in $[0,1]^2$ and so that, for any Lebesgue-measurable subset $A$ of $[0,1]$:
\begin{equation}\label{eq:uniform_marginals}
    \mu(A\times[0,1]) = \mu([0,1]\times A) = \Leb(A)
\end{equation}
where $\Leb$ is the uniform measure on $[0,1]$ (see \cite[Remark~1.2]{borga2022large} for details). 
Such measures $\mu$ with uniform projections on the axes are called \textit{permutons}\footnote{
The joint CDF of a permuton is called a \textit{copula} in the statistics literature; see \cite{gruebel2024copulas}. 
}.
Permutons are natural limit objects for large permutations, see e.g.~\cite{hoppen2013permutons,bassino2020universal}.
The associated model of random permutations $\sigmaP{\pts_\mu^n}$ will then simply be denoted by $\sigma_\mu^n$.
This is a broad generalization of the uniform measure on permutations of size $n$, which corresponds to $\mu=\Leb_{[0,1]^2}$, the Lebesgue measure on $[0,1]^2$.
Such models have been considered in the literature under various perspectives, see e.g.~\cite{deuschel1995records,borga2022large,dubach2023LIS-linear,dubach2023LIS-inermediate,sjostrand2023monotone}.

In the current paper, we are interested in the binary search tree $\lbst{\sigma_\mu^n}$ of this random permutation model.
Since we will be interested only in the shape of this tree (height, subtree size convergence), we may and will equivalently consider the tree $\lbst{\pts_\mu^n}$ instead of $\lbst{\sigma_\mu^n}$.

\subsection{First main result: universal behavior of the BST height}
\label{ssec:result-height}

For a (labeled binary) tree $\mathcal T$, we denote by $h(\mathcal T)$ its height, i.e.~the maximal distance from a leaf to the root.
As mentioned in \Cref{ssec:context}, Devroye~\cite{devroye1986note} proved that for uniform random permutations $\sigma^n$ of size $n$, the quantity $\height{\sigma^n}/\log n$ converges in probability and in $L^p$ (for all $p\ge1$) to a constant $c^*$, defined as the unique solution to $c\log(2e/c)=1$ with $c\geq 2$.
Our first result gives a sufficient condition on a permuton $\mu$, under which the same result holds for the height of $\lbst{\pts^n_\mu}$.
In the following, a permuton $\mu$ is said to satisfy assumption \assumption{1}~\hypertarget{A1}{if $\mu$ has a bounded density $\rho$ on the whole unit square $[0,1]^2$, which is continuous and positive on a neighborhood of $\{0\} \times [0,1]$.}

\begin{thm}[Universality of BST height for permuton samples]\label{thm:height}
    Let $\mu$ be a permuton satisfying Assumption~\assumption{1}.
    Then, as $n$ goes to infinity, the following convergence holds in probability and in $L^p$ for any $p\ge1$:
    \begin{align*}
        \frac{\height{\pts^n_\mu}}{c^*\log n} \longrightarrow1\,.
    \end{align*}
\end{thm}
The constant $c^*$ in the above theorem is the same as in the uniform case, i.e.~the unique solution to $c\log(2e/c)=1$ with $c\geq 2$.

Let us comment on the Assumption~\assumption{1}.
A natural sufficient condition is that the density $\rho$ is continuous on the whole square $[0,1]^2$ and positive on the left edge $\{0\}\times [0,1]$.
In Section~\ref{sec:hypothesis-A1}, we will see that this positivity assumption cannot be skipped.
Indeed we exhibit, for any $\delta>0$, a permuton $\mu_\delta$ with a continuous density vanishing on $\{0\}\times [\tfrac12,1]$ such that $\lbst{\pts^n_\mu}$ has height at least $\Theta(n^{(1-\delta)/2})$ with probability tending to 1.

In \Cref{ssec:logarithmic height but greater constant}, we also provide an example of a permuton $\mu_\beta$ with density $1$ on the band $[0,\beta] \times [0,1]$ and for which, for any $\eps>0$, the tree $\lbst{\pts^n_{\mu_\beta}}$ has height at least $\frac{1-\beta}{\beta+\eps}\log n$ with high probability.
Choosing $\beta$ and $\eps$ such that $(1-\beta)/(\beta+\eps) > c^*$ gives us an example of permuton which has a positive continuous density on a band $[0,\beta] \times [0,1]$, but for which the conclusion of \Cref{thm:height} fails.
Hence the existence of a density on the whole square in Assumption~\assumption{1} is needed.

On the other hand, we could not construct a permuton $\mu$
such that $\height{\pts^n_\mu}$ is asymptotically smaller than $c^* \log n$
with a non-vanishing probability.
This leads us to the following conjecture.
\begin{conj}
  \label{conj:uniform-optimal}
    For any permuton $\mu$ and $\eps >0$, one has
    \[ \lim_{n\to\infty} \prob\left[ \frac{ \height{\pts^n_\mu} }{\log n} < c^*-\eps \right] =0 .\]
\end{conj}
As evidence supporting \Cref{conj:uniform-optimal}, we show that the statement holds for permutons with a density assumed to be continuous and positive around a single point $(0,y)$ on the left edge of the unit square;
see \Cref{prop:uniform_optimal_partial}.

Let us emphasize that, given the application of binary search trees to data storage, 
this paper focuses on permutons that yield \enquote{well-packed} BSTs.
It is easy to construct permutons that yield much deeper BSTs (see e.g.~Section~\ref{sec:hypothesis-A1}), but we will not attempt to characterize them here.

\begin{remark}
There is a natural way to construct the sequence of permutations $(\sigma^n_\mu)_{n \ge 1}$ on the same probability space: 
we consider a single infinite sequence $((x_i,y_i))_{i \ge 1}$ of i.i.d.~points with law $\mu$, and construct each $\sigma^n_\mu$ using the first $n$ points of this sequence.
We may wonder whether the convergence in \Cref{thm:height} holds almost surely for this construction.
We do not know whether this is the case, even for $\mu=\Leb_{[0,1]^2}$.
It has been shown by Pittel in \cite{pittel1984BST} that, if $Y_1,Y_2,\dots$ is an infinite sequence of uniform random variables in $[0,1]$, then $h(\mathcal T\langle Y_1,\dots,Y_n\rangle)/\log n$ converges a.s.~to a constant $\beta$.
Combining with the above-mentioned result of Devroye, the constant $\beta$ must be equal to $c^*$, as noted in \cite[Section 5]{devroye1986note}.
However, the sequence $\left( h(\mathcal T\langle Y_1,\dots,Y_n\rangle) \right)_{n\ge1}$ considered by Pittel does not have the same distribution as $\left( \height{\pts^n_\mu} \right)_{n\ge1}$ for $\mu=\Leb_{[0,1]^2}$.
Indeed, while both have the same distribution for each $n$, the couplings between different values of $n$ are different:
in Pittel's model, the new element $Y_{n+1}$ is always added at the end of the sequence $Y_{1},\dots, Y_n$, while in our model, the new element of $Y_{(1)},\dots,Y_{(n+1)}$ is added in a uniformly random position in $Y_{(1)},\dots,Y_{(n)}$ ($Y_{(i)}$ depends implicitly on $n$, not only on $i$).
\end{remark}

\subsection{Second main result: subtree size convergence of the BSTs}
\label{ssec:result-stscv}

In this section, we state a limit theorem for $\lbst{\pts^n_\mu}$ (under a mild assumption on $\mu$) in the sense of the subtree size convergence recently introduced by Gr\"ubel~\cite{grubel2023note}.
We first recall this notion of convergence.
\medskip
 
From now on, we identify nodes in a binary tree with finite words in the alphabet $\{0,1\}$ as follows:
the empty word $\varnothing$ corresponds to the root, and for a node $v$ encoded by $w$,
the words $w\lChild$ and $w\rChild$ obtained by appending $\lChild$ or $\rChild$ to $w$ encode respectively the left and right children of $v$.
Moreover, we let $\rV=\{\lChild,\rChild\}^*$ be the set of all finite words on $\{\lChild,\rChild\}$, representing all nodes of the complete infinite binary tree.
With this notation, a labeled tree is identified with a function from a subset of $\rV$ to $\bR$, where the domain of the function is the set of nodes in the tree, and a node is mapped to its label.
In particular, $\mathcal T(v)$ denotes the label of the node $v$ in $\mathcal T$.
We also write $v \in\mathcal T$ to indicate that the node $v$ is in the tree $\mathcal T$.

Given a finite (potentially labeled) tree $\lT$ and a node $v\in\rV$, we let
\begin{align*}
    t(\lT,v):=\frac{1}{|\lT|}\Big|\Big\{u\in\lT:v\preceq u\Big\}\Big|,
\end{align*}
where $v \preceq u$ means that $v$ is a prefix of $u$.
In words, $t(\lT,v)$ is the proportion of nodes in $\lT$ which are descendants of $v$.

Further write $\Psi$ for the set of functions $\psi:\rV\rightarrow[0,1]$ such that $\psi(\Troot)=1$ and such that, for any $v\in\rV$, we have $\psi(v)=\psi(v\lChild)+\psi(v\rChild)$.
Then a sequence of binary trees $(\lT^n)_{n\in\bN}$ is said to converge to a function $\psi\in\Psi$ if and only if, for any $v \in \rV$, the quantity $t(\lT^n,v)$ converges to $\psi(v)$.
If that is the case, we write
$
    \lT_n\grubeling\psi
$,
and refer to this as \textit{subtree size convergence} and to $\psi$ as the \textit{subtree size limit} of $\lT_n$.
\medskip

We now define two important objects for the subtree size convergence of BSTs of permuton samples.
For any complete infinite BST $\lT:\rV\rightarrow(0,1)$ with labels in $(0,1)$,
we define $\lTleft:\rV\rightarrow\bR$ as follows.
First of all, if $v$ consists only of zeros, we let $\lTleft(v)=0$.
Now, given that $v=v'\rChild\lChild^k$ for some $k\geq 0$, let $\lTleft(v)=\lT(v')$.
Informally, $\lTleft(v)$ is the right-most ancestor of $v$ to its left.
Define similarly $\lTright$ such that $\lTright(v)=1$ if $v$ consists only of ones, and $\lTright(v)=\lT(v')$ whenever $v=v'\lChild\rChild^k$ for some $k\geq0$.
In words, $\lTright(v)$ is the left-most ancestor of $v$ to its right.
See \Cref{fig:right-left-ancestors} for an illustration of $\lTleft$ and $\lTright$.
We note that these definitions imply that $\lTleft(v)<\lT(v)<\lTright(v)$ for any $v\in\rV$.

\begin{figure}[htb]
    \centering
    \includegraphics[scale=0.5]{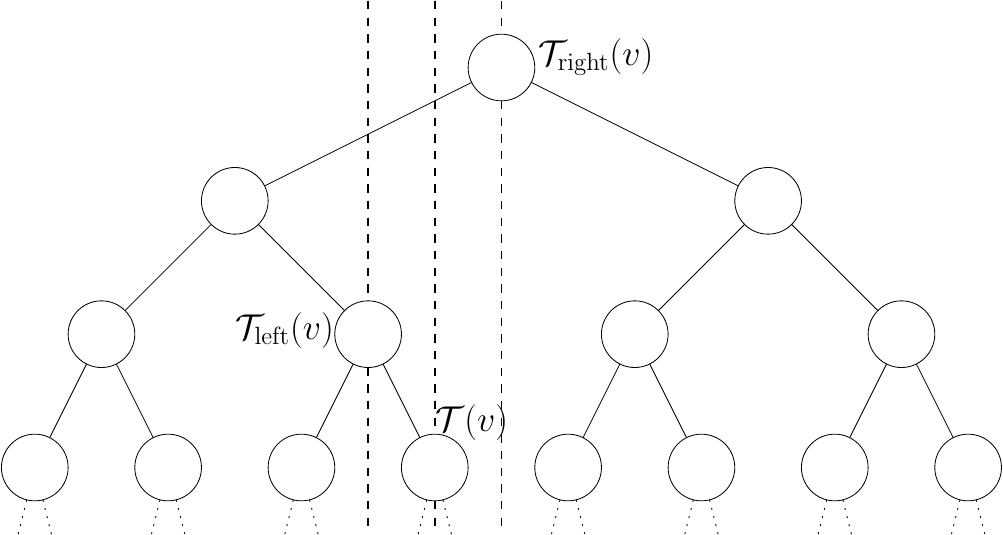}
    \caption{Representation of $\lTright$ and $\lTleft$ on a labeled BST, for the node $v=011$.}
    \label{fig:right-left-ancestors}
\end{figure}

Given a probability measure $m$ on $[0,1]$ without atoms, write $\Tlim{m}\in\Psi$ for the following random object.
First, let $Y=(Y_1,Y_2,\ldots)$ be an infinite sequence of independent random variables distributed according to $m$ and write $\lT^m=\lbst{Y}$ for the corresponding (infinite) BST.
Then, let $\Tlim{m}=\lTright^m-\lTleft^m$.
This is well-defined, since $\lT^m$ is complete with probability $1$ (this follows from \cite[Theorem~6.1]{devroye1986note} and the fact that the shape of $\lT^m=\lbst{Y}$ is independent of $m$).
Further note that this object indeed belongs a.s.~to $\Psi$, since $\lTleft(v\lChild)=\lTleft(v)$, $\lTleft(v\rChild)=\lT(v)$, $\lTright(v\lChild)=\lT(v)$, and $\lTright(v\rChild)=\lTright(v)$.

To illustrate this construction, take $m=2x dx$, where $dx$ is the Lebesgue measure on $[0,1]$.
Here is an i.i.d.~sample of size 10 from $m$: 
$(0.73,0.33,0.75,0.35,0.68,0.28,0.72,0.87,0.25,0.67)$.
Its associated BST, which is the top part of the BST $\lT^m=\lbst{Y}$ associated with an infinite sample, is given in \Cref{fig:example_Psi}, left.
The associated realization of the function $\Tlim{m}$ is then given in \Cref{fig:example_Psi}, right.
For instance, for the node $v=\lChild\rChild\rChild$ as chosen in \Cref{fig:right-left-ancestors}, we can compute $\Tlim{m}(v)=\lTright^m(v)-\lTleft^m(v)=\lT^m(\Troot)-\lT^m(\lChild\rChild)=0.73-0.35=0.38$, as written in \Cref{fig:example_Psi}.

\begin{figure}[htb]
    \[
  \begin{array}{c}
 \includegraphics[scale=0.45]{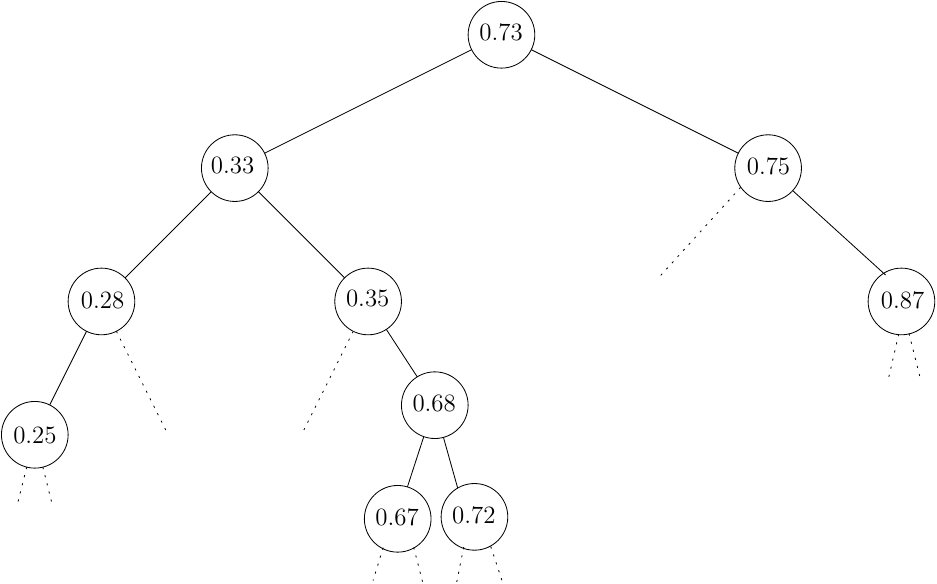} 
  \end{array} \ \begin{array}{c}
  \includegraphics[scale=0.45]{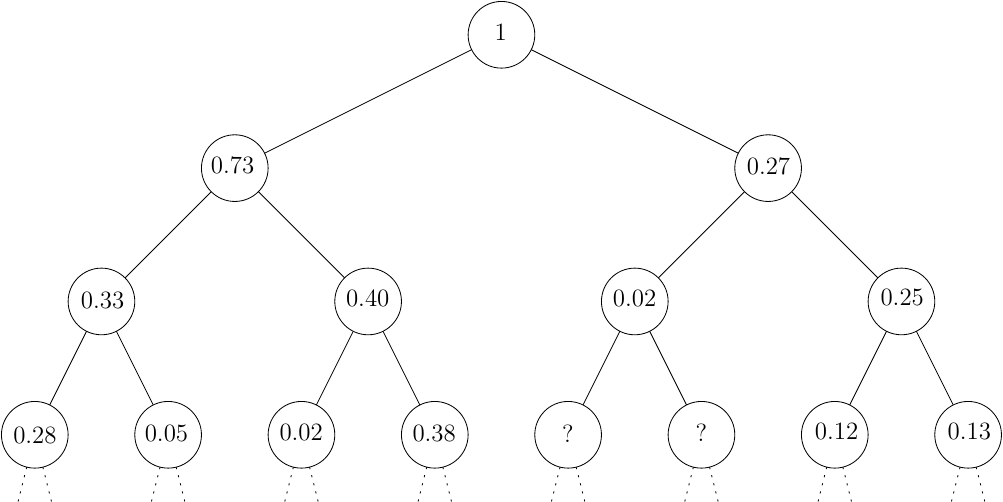}
  \end{array}
  \]
    \caption{Example of realizations of $\lT^m$ and $\Tlim{m}$ with $m=2x dx$. 
    Note that we do not have enough data to compute two of the values of $\Tlim{m}$ on nodes in the third level;
    those nodes are marked with question marks.}
    \label{fig:example_Psi}
\end{figure}
\medskip

We can now state our second main result. 
A permuton is said to satisfy Assumption \assumption{2}~\hypertarget{A2}{if there exists a probability measure $\mu_0$ on $[0,1]$ {\em without atoms} such that
\begin{equation}\label{eq:def_mu0}
    \frac1x \mu\big([0,x]\times\cdot\big)
    \,\underset{x\to0^+}{\longrightarrow}\,
    \mu_0,
\end{equation}
where the convergence occurs according to the weak topology of probability measures on $[0,1]$.} 
Assumption~\assumption{2} is weaker than \assumption{1}: 
in particular, \assumption{2} holds whenever $\mu$ admits a continuous density on a neighborhood of $\{0\} \times [0,1]$, without any further assumption on that density.

When we refer to \eqref{eq:def_mu0} instead of \assumption{2}, the measure $\mu_0$ may \textit{a priori} have atoms.
It is easy to find permutons for which \assumption{2} does not hold, and in \Cref{ssec:no_STS} we exhibit a permuton for which \eqref{eq:def_mu0} does not hold.

\begin{thm}[subtree size convergence of BSTs of permuton samples]\label{thm:limit}
    Let $\mu$ be a permuton satisfying \assumption{2} and let $\mu_0$ be defined by \eqref{eq:def_mu0}.
    Then, we have the following convergence in distribution for the subtree size topology:
    \begin{align*}
        \lbst{\pts^n_\mu} \grubeling \Tlim{\mu_0} \,.
    \end{align*}
    Conversely, if $\lbst{\pts^n_\mu}$ converges in distribution for the subtree size topology as $n\to\infty$, then there exists a probability measure $\mu_0$ on $[0,1]$ (possibly with atoms) for which \eqref{eq:def_mu0} holds, that is:
    \begin{equation*}
        \frac1x \mu\big([0,x]\times\cdot\big)
        \,\underset{x\to0^+}{\longrightarrow}\,
        \mu_0 \,.
    \end{equation*}
\end{thm}

It is interesting to see that under~\assumption{2} the limit depends on $\mu$ only through $\mu_0$.
The assumption that $\mu_0$ does not have atoms is important when identifying this limit.
A first difficulty when $\mu_0$ has some atom is that the BST $\lbst{Y_1,Y_2,\dots}$ where $Y_1$, $Y_2$, \dots, are i.i.d.~random variables with distribution $\mu_0$ is ill-defined since some of the $Y_i$'s are equal.
We can also see that, in this case, the limit of $\lbst{\pts^n_{\mu}}$ may not depend only on $\mu_0$.
Indeed, consider the permuton $\mu^1$ (resp.~$\mu^2$) supported by the set $y \equiv \tfrac12+x$ modulo 1 (resp.~$y \equiv \tfrac12-x$ modulo 1).
They both satisfy  \eqref{eq:def_mu0} with $\mu_0=\delta_{1/2}$.
But it is easy to see that the trees $\lbst{\pts^n_{\mu^1}}$ and $\lbst{\pts^n_{\mu^2}}$ have different limits in the sense of subtree size convergence: 
in particular, one has
\[\lim_{n \to \infty} t(\lbst{\pts^n_{\mu^1}},11)=\tfrac12,\text{ while }\lim_{n \to \infty} t(\lbst{\pts^n_{\mu^2}},11)=0.\]

It is natural to ask whether $\lbst{\pts^n_\mu}$ converges for the subtree size topology as soon as there exists a probability measure $\mu_0$ on $[0,1]$ (possibly with atoms) for which \eqref{eq:def_mu0} holds.
In \Cref{ssec:no_STS} we disprove this, by finding a permuton which satisfies \eqref{eq:def_mu0} with $\mu_0 = \delta_{1/2}$, although $\lbst{\pts^n_\mu}$ does not converge for the subtree size topology.
Thus \assumption{2} is useful not only to identify the subtree size limit, but to guarantee its existence as well.

Like \Cref{thm:height}, \Cref{thm:limit} can be seen as a universality result:
for any permuton satisfying \assumption{2} with $\mu_0=\Leb_{[0,1]}$, the associated BST $\lbst{\pts^n_\mu}$ has the same subtree size limit as that of the BST of a uniform random permutation (see \Cref{ssec:logarithmic height but greater constant} for a non-uniform permuton with $\mu_0$ uniform).
We remark that none of the universality classes for the height or for the subtree size convergence is larger than the other in the following sense:
\begin{itemize}
    \item There are permutons $\mu$ for which $\lbst{\pts^n_\mu}$ has asymptotically the same height as in the uniform case, but a different subtree size limit. 
    An example is the \enquote{Mallows permuton}, as discussed in \Cref{ssec:example_Mallows}.
    \item There are permutons for which $h(\lbst{\pts^n_\mu})$ is asymptotically larger than the uniform case, but the subtree size limits are the same. 
    An example is given in \Cref{ssec:logarithmic height but greater constant}.
\end{itemize}

\begin{remark}
As explained in~\cite[Lemma 1]{grubel2023note}, elements from $\Psi$ are in correspondence with probability measures on the set $\mathbb V_\infty$ of infinite binary words (with the usual $\sigma$-algebra spanned by cylinders):
if $\psi$ is in $\Psi$, we simply associate to it the measure $\nu$ defined by $\nu(B_u)=\psi(u)$ for all $u \in \mathbb V$, where $B_u$ is the set of infinite words starting with $u$. 
Equivalently, one can construct a random word $v$ with law $\nu$ by setting
\begin{align}\label{eq:random sample}
    \prob\big[ v_{k+1}=\lChild~\big|~v_1,\ldots,v_k \big] =\frac{\psi(v_1\cdots v_k\lChild)}{\psi(v_1\cdots v_k)}\,.
\end{align}

Since the limit $\psi_{\mu_0}$ in \Cref{thm:limit} is a {\em random} element of $\Psi$, it can be seen as a {\em random} probability measure $\nu_{\mu_0}$ on $\mathbb V_\infty$.
This random measure is the conditional law of the random word $v$ defined by \eqref{eq:random sample} with $\psi=\psi_{\mu_0}$, given $\psi_{\mu_0}$, or equivalently given the tree $\lT^{\mu_0}$.
Taking a uniform random variable $X$ in $[0,1]$, the right-hand side of \eqref{eq:random sample} is the probability that $X$ is smaller than $\lT^{\mu_0}(v_1,\dots,v_k)$ knowing that it is between $\lTleft^{\mu_0}(v_1\cdots v_k)$ and $\lTright^{\mu_0}(v_1\cdots v_k)$.
Consequently, the random word $v$ corresponds to the infinite insertion procedure of $X$ in $\lT^{\mu_0}$
(since $\lT^{\mu_0}$ is an infinite complete binary tree, if we try to insert $X$ in $\lT^{\mu_0}$, the procedure never stops, but it yields an infinite word of $\lChild$ and $\rChild$ recording whether we visit the left or the right subtree of each node).

We can also construct the random word $v$ without knowing $\lT^{\mu_0}$ as follows.
Informally, the idea is that we build only the branch of $\lT^{\mu_0}$, which is visited in the insertion procedure of $X$.
First set $a_0=0$ and $b_0=1$.
Then for $k \ge 1$, given $a_{k-1},b_{k-1}$, we sample $Y_k$ according to $\mu_0$ conditioned on being in $(a_{k-1},b_{k-1})$ and let $v_k,a_k,b_k$ be as follows:
\begin{enumerate}[(i)]
    \item if $X \le Y_k$ then $v_k=\lChild$, $a_k=a_{k-1}$ and $b_k=Y_k$;
    \item if $X > Y_k$ then $v_k=\rChild$, $a_k=Y_{k}$ and $b_k=b_{k-1}$.
\end{enumerate}
Note that, for each $k$, it holds that $X$ is between $a_k$ and $b_k$, and we have no further information on $X$ at step $k$.
Hence, instead of sampling $X$ in advance, we can alternatively choose at step $k$ item (i) with probability $(Y_k-a_{k-1})/(b_{k-1}-a_{k-1})$, and item (ii) otherwise.

This procedure constructs $v$ without knowing $\lT^{\mu_0}$, so it does not give access to $\nu_{\mu_0}$, which is the conditional law of $v$ knowing $\lT^{\mu_0}$. 
It only gives access to the intensity measure\footnote{
For a random measure $\zeta$ on a measurable space $(S,\mathcal A)$, its {\em intensity measure} $\mathbb E \zeta$ is defined by $(\mathbb E \zeta) (A)=
\mathbb E [\zeta (A)]$ for all $A \in \mathcal A$.} $\mathbb E \nu_{\mu_0}$, which is the (unconditional) law of $v$.
In particular, \Cref{thm:limit} implies that, for any $u$ in $\mathbb V$,
\[ \lim_{n \to +\infty} \mathbb E\big[ t(\lbst{\pts_\mu^n},u) \big] 
= \mathbb E\big[ \psi_{\mu_0}(u) \big] 
= \mathbb E \nu_{\mu_0} (B_u) 
= \mathbb P[v \in B_u] ,\]
where $v$ is the above random word.
\end{remark}

\subsection{Decomposition of BSTs and proof strategies}
\label{ssec:proof-overview}

In this section we present a useful decomposition of the BSTs drawn from permuton samples, 
and provide an overview of the proofs of Theorems~\ref{thm:height} and \ref{thm:limit}.

\subsubsection*{Decomposing a BST from a permuton sample}

A basic idea in this work consists in decomposing the BST drawn from a permuton sample, as a top tree to which hanging trees are attached.
To this end, we first consider the $K$ left-most points in the sample $\pts^n_{\mu}$ 
(i.e.~the $K$ points with smallest $x$-coordinates).
These $K$ points are the first ones to be inserted in the construction of $\lbst{\pts^n_{\mu}}$ and are therefore inserted at the top of the tree.
We will refer throughout the paper to the part of $\lbst{\pts^n_{\mu}}$ corresponding to these first $K$ points as the {\em top tree}.
The labels in the top tree correspond to the $y$-coordinates $y_{(1)},\dots,y_{(K)}$ of these first $K$ points.
Now, consider the subdivision $I_1,\dots,I_{K+1}$ of $[0,1]$ induced by these numbers $y_{(1)},\dots,y_{(K)}$.
In the construction of $\lbst{\pts^n_{\mu}}$, further points $(x,y)$ will be inserted between some pair of consecutive vertices in the top tree, depending on the index $j$ such that $y \in I_j$.
Hence the tree $\lbst{\pts^n_{\mu}}$ is obtained by grafting to the top tree
one subtree for each interval $I_j$; see Figure~\ref{fig:bands}.
These trees will be refered to as the {\em hanging trees}.

\begin{figure}[htb]
    \centering
    \includegraphics[scale=0.55]{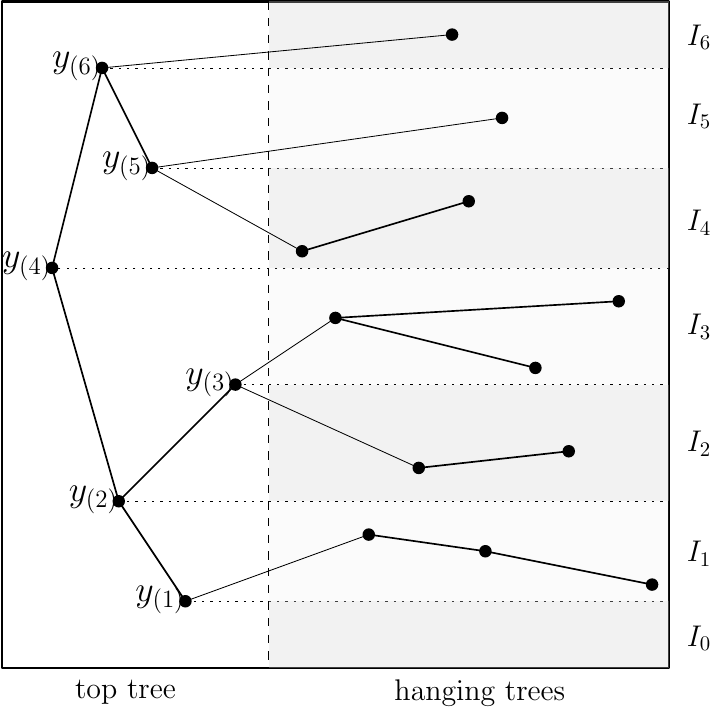}
    \caption{A sample of points and its associated BST, decomposed as top and hanging trees
    (for $K=6$).
    The BST has been rotated of 90 degrees to the left, so that it can be drawn directly on
    the set of points.}
    \label{fig:bands}
\end{figure}

\subsubsection*{Subtree size convergence}

From this decomposition, the proof of the subtree size convergence is relatively simple.
We take $K$ large, but independent of $n$.
First of all we prove that, under the regularity Assumption~\assumption{2}, 
the first $K$ points look like i.i.d.~random variables sampled from the left distribution $\mu_0$ (see Proposition~\ref{prop:k points} and Corollary~\ref{corol:k-pts-fixed-n}).
Moreover, since a permuton $\mu$ has by definition uniform projections, the proportion of points in each horizontal band seen in Figure~\ref{fig:bands} is asymptotically given by the height of the band.
By choosing $K$ large enough, this gives us the proportion of nodes in all subtrees until any given depth, thus proving the subtree size convergence 
(see Lemma~\ref{lem:subtree-proporitions_in_BSTs} and the proof of Theorem~\ref{thm:limit} thereafter).

\subsubsection*{Height universality}

To prove Theorem~\ref{thm:height}, we use the same decomposition of the BST into a top tree and hanging trees, but this time we take $K=\beta n + {\o_\prob(n)}$ with $\beta$ small but independent from $n$.
The heights of the top tree and hanging trees are controlled via different approaches.

For the top tree, our Assumption~\assumption{1} (specifically the continuity and positivity of the density near the left edge) ensures the following: 
for $\beta$ small enough, the density restricted to the left vertical band $[0,\beta] \times [0,1]$ is close, up to a multiplicative factor, to a density depending only on the $y$-coordinate.
It is easy to see that if the density only depended on $y$, then the associated permutation would be uniformly distributed, and thus the BST would have height $(c^*+o_\prob(1))\log K=(c^*+o_\prob(1))\log n$.
We then use comparison arguments to prove that the top tree also has height $(c^*+o_\prob(1))\log n$ under Assumption~\assumption{1}.
Such comparison arguments need to be handled carefully, since adding a single point to a point set can halve the height of the associated BST (see Remark~\ref{rem:height and points}).
On the other hand, we show that removing a point cannot decrease this height by more than $1$ (Lemma~\ref{lem:chain trick}), and use this to control the height of the top tree.

It remains to argue that the hanging trees all have height $o_\prob(\log n)$.
For $K=\beta n + {\o_\prob(n)}$, the horizontal bands in Figure~\ref{fig:bands} contain $\O(1)$ points on average, but the largest number of points in a band is actually $\O(\log n)$ (see Proposition~\ref{prop:tight rows}).
The hanging trees are themselves BSTs of random point sets, and therefore we expect them to have height $\O(\log \log n) \ll \log n$.
However, there are $\O(n)$ hanging trees, and we {need} all of them to have height $o(\log n)$.
To prove this,we need good deviation estimates on the fact that the BST of a point set has a height negligible compared to its size.
Such estimates are provided by Devroye for uniform BSTs \cite[Lemma 3.1]{devroye1986note}, but the monotonicity properties of BSTs are not good enough to use direct comparison arguments.
We solve this difficulty by comparing, for any point set, the height of its BST and the length of its longest monotone subsequences, and then by using monotonicity properties and deviation bounds for the latter (see Lemma~\ref{lem: upper bound LIS LDS} and Corollary~\ref{cor: extreme deviation BST height}).

\subsubsection*{Fixed number of points and Poisson point processes}

As in many results considering point sets, it is often more convenient to work with a Poisson number $N\sim\poisson{n}$ of points, instead of a fixed number $n$. 
Indeed, the point set then becomes a Poisson point process on $[0,1]^2$ with intensity $n\mu$ and gains useful independence properties:
conditionally given the labels $y_{(1)},\dots,y_{(k)}$ in the top tree, the hanging trees are independent from each other.

In Theorem~\ref{thm:Poisson fixed} we provide a general de-Poissonization result, relating the asymptotic height of a BST constructed from a Poisson point process with intensity $n\mu$ to that of $n$ i.i.d.~points with law $\mu$.
The proof of this result uses the comparison lemma already mentioned above (Lemma~\ref{lem:chain trick}), and is made difficult by the fact that we only have a {control} in one direction (recall that adding a single point to a point set can halve the height of the associated BST).

\subsection{Basic probabilistic facts and notation}
Throughout the paper, \enquote{with high probability} (or w.h.p.~for short) means \enquote{with probability tending to $1$, as $n$ tends to $\infty$}.
We also use the notation $X_n=o_\prob(Y_n)$ to say that $X_n/Y_n$ converges to $0$ in probability.

We start by stating a useful lemma, which compares the size of random subsets to binomial random variables.
We write $X\preceq Y$ (resp.~$X\succeq Y$) to denote that $X$ is stochastically smaller (resp.~larger) than $Y$.

\begin{lemma}\label{lem:size dom}
    Let $m$ be an integer and let $i,j< m$. 
    Let $I,J$ be subsets of $\{1,\dots,m\}$ of respective sizes $i,j$, where $I$ is fixed and $J$ is uniformly random. 
    Then the following stochastic domination holds:
    \begin{equation*}
        |I\cap J| \preceq \binomial{i}{\frac{j}{m-i}} 
    \end{equation*}
    where by convention the law $\binomial{n}{p}$ with $p>1$ is the Dirac law at $n$.
\end{lemma}

\begin{proof}
    By symmetry, $ |I\cap J|$ would have the same distribution if $J$ were fixed and $I$ taken uniformly at random among subsets of size $i$ of $\{1,\dots,m\}$.
    In this situation, $I$ can be constructed by uniformly picking $i$ elements of $\{1,\dots,m\}$ without replacement.
    Each picked element has probability at most $\frac{j}{m-i}$ of being in $J$, which proves the lemma.
\end{proof}

Next, we record some well-known asymptotic estimates for Poisson random variables with large parameters.
If $N$ is a $\poisson{n}$ distributed random variable, then $N/n$ converges to $1$ in distribution and in moments: 
i.e., for any fixed integer $p \ge 1$, we have
\begin{equation}\label{eq:moment_poisson}
\expec[N^{p}]=n^{p}(1+o(1)).
\end{equation}
Furthermore, we have the following tail estimates, easily obtained via Chernoff bounds.
\begin{lemma}\label{lem:Poisson_bounds}
    Let $\lambda>0$ and $X$ be a $\poisson{\lambda}$ distributed random variable.
    Then, for $x \ge \lambda$, we have
    \[ \prob[X \ge x] \le \left(\frac{e\lambda}{x}\right)^xe^{-\lambda}, \]
    while, for $x \le \lambda$, it holds that
    \[ \prob[X \le x] \le \left(\frac{e\lambda}{x}\right)^xe^{-\lambda}. \] 
\end{lemma}

\begin{proof}
Assume $x \ge \lambda$. For $\theta >0$, we have
\[\prob[X \ge x] \le \frac{\expec[e^{\theta X}]}{e^{\theta x}} =e^{\lambda e^{\theta} -\lambda -\theta x},\]
and the first inequality in the lemma follows by setting $e^\theta=x/\lambda$. 
The second one is proved similarly.
\end{proof}

We conclude this section with a variant of the classical Glivenko--Cantelli theorem for triangular arrays.

\begin{prop}\label{prop:glivenko_cantelli_triangular}
Let $\mu$ be a probability measure with distribution function $F(x):=\mu((-\infty,x])$, and a finite fourth moment.
For each $n \ge 1$, let $(X_{i,n})_{1 \le i \le n}$ be i.i.d.~random variables with common distribution $\mu$ and let 
\[F_n(x) := \frac1n\, \big|\{i \le n: \, X_{i,n} \le x\}\big|\]
be their {\em empirical distribution function}. Then, a.s.~it holds that $F_n$ converges uniformly to $F$.
\end{prop}

\begin{proof}
A classical fourth moment computation, together with Borel--Cantelli lemma 
--- see e.g.~\cite[Theorem 6.1]{billingsley2012probability_and_measure} --- 
shows that, for any fixed $x$, $F_n(x)$ converges a.s.~to $F(x)$.
The rest of the proof is similar to that of the classical Glivenko--Cantelli theorem which considers a single sequence $X_i$ of i.i.d.~random variables instead of a triangular array, but does not require a fourth moment condition;
see e.g.~\cite[Theorem 20.6]{billingsley2012probability_and_measure}.
\end{proof}

\section{Subtree size convergence}
\label{sec:ssc-proof}

\subsection{Convergence of the first elements}

In the following proposition, $\mu$ is a permuton and $\mu_0$ a measure on $[0,1]$.
Note that, at this stage, we do not need to assume that $\mu_0$ has no atoms as in~\assumption{2}.
For each $n$, we let $N$ be a $\poisson{n}$ distributed random variable and $\pts^N= \left\{ \left(X^N_i,Y^N_i\right) \right\}_{i\le N}$ be i.i.d.~random variables with distribution $\mu$ (we drop the subscript $\mu$ for clarity);
equivalently, $\pts^N$ is a Poisson point process on $[0,1]^2$ with intensity $n \mu$.
Also let $\left( \big(X^N_{(i)},Y^N_{(i)}\big) \right)_{i\le N}$ be its reordering such that $X^N_{(1)} < \dots < X^N_{(N)}$.

\begin{prop}\label{prop:k points}
    The following statements are equivalent:
    \begin{enumerate}[(i)]
    \item \label{item:cv_to_mu0} we have the weak convergence of measures        
    \begin{equation*}
        \lim_{x\rightarrow0^+} \, \tfrac{1}{x}\, \mu\big([0,x]\times\cdot\big) = \mu_0\,;
    \end{equation*}
    \item \label{item:cv_few_first_Y} for any fixed $K \ge 1$, we have the following convergence in distribution in $\mathbb R^K$:
    \begin{equation*}
        \left( Y_{(1)}^N, \dots, Y_{(K)}^N \right) \underset{n\rightarrow\infty}{\longrightarrow}(Y_k)_{1\leq k\leq K}\,,
    \end{equation*}
    where $(Y_k)_{1\leq k\leq K}$ is a sequence of $K$ i.i.d.~random variables distributed according to $\mu_0$.
    \item \label{item:cv_first_Y} the random variable $Y_{(1)}^N$ converges in distribution to a random variable $Y_1$ with law $\mu_0$.
    \end{enumerate}
\end{prop}

\begin{proof}
    We first prove that \eqref{item:cv_to_mu0} implies \eqref{item:cv_few_first_Y}.
    Let $I_1,\dots,I_K$ be intervals of $[0,1]$ whose boundaries contain no atom of $\mu_0$.
    We aim to show that:
    \begin{equation}\label{cv_voulue}
        \prob\left[ \forall k\leq K,\; Y_{(k)}^N\in I_k \right]
        \underset{n\rightarrow\infty}{\longrightarrow}
        \mu_0(I_1)\dots\mu_0(I_K).
    \end{equation}
    Fix $\eps>0$ and set $L := \lfloor 1/{\eps^3} \rfloor$.
    By taking $\eps$ small enough, we can assume that $L\ge K$.
    Using the assumption \eqref{eq:def_mu0} and the Portmanteau theorem, there exists $x_0>0$ such that for any $x\leq x_0$ and any $k\le K$:
    \begin{equation}\label{controle_cvloi}
        \left| \frac{1}{x}\mu\big( [0,x)\times I_k \big)
        - \mu_0\left( I_k \right) \right| \leq \eps / L\,,
    \end{equation}
    Set $n_0 := \lceil \frac{1}{\eps x_0} \rceil$ and consider any $n\geq n_0$.
    For all $0\le i\le L-1$ and $1\le k\le K$, define the following  blocks and columns:
    \begin{equation*}
        B_{i,k} := \left[ \frac{i}{L\eps n} , \frac{i+1}{L\eps n} \right) \times I_k
        \quad\text{and}\quad
        C_i :=  \left[ \frac{i}{L\eps n} , \frac{i+1}{L\eps n} \right) \times [0,1].
    \end{equation*}
    Using $\frac{1}{\eps n}\leq x_0$ and \eqref{controle_cvloi}, we have for each $0\le i\le L-1$ and $1\le k\le K$:
    \begin{align}\label{approx_box_mass}
        &\left| \mu\left( B_{i,k} \right) - \frac{1}{L\eps n}\mu_0(I_k) \right|
        \\\nonumber&\leq \left| \mu\left( \left[ 0 , \frac{i+1}{L\eps n} \right) \times I_k \right) - \frac{i+1}{L\eps n}\mu_0(I_k) \right|
        + \left| \frac{i}{L\eps n}\mu_0(I_k) - \mu\left( \left[ 0 , \frac{i}{L\eps n} \right) \times I_k \right) \right|
        \\&\nonumber \leq \frac{i+1}{L\eps n}\frac{\eps}{L} + \frac{i}{L\eps n}\frac{\eps}{L}
        \\& \leq \frac{2}{n L}.\nonumber
    \end{align}
    Now define the events
    \begin{equation*}
        E_1 := \Big\{ \forall i\leq L-1,\; \left| \pts^N \cap C_i \right| \leq 1 \Big\}
        \quad\text{and}\quad
        E_2 := \left\{ \left| \bigcup_{0\leq i\leq L-1}C_i \cap \pts^N \right| \geq K \right\}.
    \end{equation*}
    In words, $E_1$ is the event that {each one of} the first $L$ columns of width $1/L\eps n$ contains at most one point, while $E_2$ is the event that there are at least $K$ points in total in all those columns, that is in the set $[0,1/\eps n)\times[0,1]$.
    Each $\left| \pts^N \cap C_i \right|$ follows a Poisson law of parameter $n\mu(C_i)=\frac{1}{L\eps}$, thus 
    \begin{align*}
        1-\prob[E_1] = \prob\Big[ \exists i\leq L-1,\; |\pts^N\cap C_i| \geq 2 \Big]
        \leq L \left( 1 - e^{-\frac{1}{L\eps}} - \frac{1}{L\eps}e^{-\frac{1}{L\eps}} \right)\,.
    \end{align*}
    Recalling that $L=\lfloor1/\eps^3\rfloor$, it follows that
    \begin{align*}
        \prob[E_1]\geq1-\O\left(\frac{1}{L\eps^2}\right)=1-\O(\eps)\,,
    \end{align*}
    where the constant in $\O(\cdot)$ is  independent of $n$ and $\eps$.
    Likewise, $\left|\left(\cup_{i}C_i\right)\cap \pts^N \right|$ follows a Poisson law of parameter $1/\eps$, so $\prob[E_2]\geq1-o(1)$ as $\eps\rightarrow0$, for any fixed $K$ and uniformly in $n$.
    We conclude that the event $E:=E_1 \cap E_2$ satisfies $\prob[E] \geq 1-\delta(\eps)$,
    where $\lim_{\eps\rightarrow0}\delta(\eps)=0$ and $\delta$ does not depend on $n$.
    
    Under the event $E$, each column $C_i$ contains at most one point, but all the columns together contain at least $K$ points.
    This means that the column indices $i_1,\dots,i_K$ of the $K$ left-most points of $\pts^N$ satisfy $0\leq i_1<\ldots<i_K\leq L-1$.
    Thus, we get
    \begin{equation*}
        \prob\left[ \left\{ \forall k\leq K,\; Y_{(k)}^N\in I_k \right\} \cap E \right]
        =\sum_{0\leq i_1<\dots<i_K\leq L-1} \prob\left[ \left\{ \forall k\leq K ,\; \left( X_{(k)}^N,Y_{(k)}^N \right) \in B_{i_k,k} \right\} \cap E \right] .
    \end{equation*}
    The latter event is equivalent to the conjunction of four facts:
        i) for each $k\le K$, $B_{i_k,k}$ contains exactly $1$ point from $\pts^N$;
        ii) for each $k\le K$, the remainder $C_{i_k} \setminus B_{i_k,k}$ of the column $C_{i_k}$ contains no point;
        iii) the columns $C_i$ with $i \le i_K$ and $i \notin \{i_1,\dots,i_K\}$ are empty; and,
        iv) each column $C_{i}$ with $i_K<i \le L-1$ contains at most $1$ point.
   In a Poisson point process, the number of points in disjoint sets are independent Poisson random variables, so we get
    \begin{multline*}
        \prob\left[ \left\{ \forall k\leq K,\; Y_{(k)}^N\in I_k \right\} \cap E \right]
         = \sum_{i_1<\dots<i_K} 
        \Bigg[\prod_{k=1}^K n\mu\left(B_{i_k,k}\right)e^{-n\mu(B_{i_k,k})} e^{-n(\mu(C_{i_k})- \mu(B_{i_k,k}) )}  \\
        \cdot 
        \prod_{\substack{i \le i_K \\
        i \notin \{i_1,\dots,i_K\}}} e^{-n\mu(C_i)} 
        \prod_{i=i_K+1}^{L-1} e^{-n\mu(C_i)} \big(1 + n\mu(C_i)\big)\Bigg]\,.
        \end{multline*}
    Combining all the exponential terms simplifies to $e^{-n \mu([0,\frac{1}{n \eps})\times [0,1])} = e^{-1/\eps}$.
    Moreover, for each $i>i_K$, we have $\mu(C_i)=\frac1{L\eps n}$. Therefore, we get
    \begin{align*}
    \prob\left[ \left\{ \forall k\leq K,\; Y_{(k)}^N\in I_k \right\} \cap E \right] = n^K e^{-1/\eps}
    \sum_{i_1<\dots< i_K}
    \left( 1 + \frac{1}{L\eps} \right)^{L-i_K-1}\prod_{k=1}^K \mu\left(B_{i_k,k}\right)\,.
    \end{align*}
    Using \eqref{approx_box_mass}, we deduce that
    \begin{align}\label{upper_boxes_proba}
      & \prob\left[ \left\{ \forall k\leq K,\; Y_{(k)}^N\in I_k \right\} \cap E \right] 
        \\\nonumber&\leq n^K e^{-1/\eps}
        \sum_{i_1<\dots< i_K}
        \left( 1 + \frac{1}{L\eps} \right)^{L-i_K-1}\prod_{k=1}^K\left(\frac{1}{L\eps n}\mu_0(I_k)+\frac{2}{nL}\right)
        \\\nonumber&\leq c(\eps) \prod_{k=1}^K \Big( \mu_0(I_k) + 2\eps \Big)
    \end{align}
    where
    \begin{equation}\label{eq:def c}
        c(\eps):= \left(\frac{1}{L\eps}\right)^K e^{-1/\eps} \sum_{i_K=K-1}^{L-1} 
        \left( 1 + \frac{1}{L\eps} \right)^{L-i_K-1}
        \genfrac(){0pt}{}{i_K}{K-1}.
    \end{equation}
   Likewise
    \begin{align}\label{lower_boxes_proba}
        \prob\left[ \left\{ \forall k\leq K,\; Y_{(k)}^N\in I_k \right\} \cap E \right] 
        \geq c(\eps) \prod_{k=1}^K \Big(\big(\mu_0(I_k) - 2\eps\big)_+\Big),
    \end{align}
    where we use the notation $x_+=\max(x,0)$. 
    Moreover, this reasoning can be applied with each $I_k$ being the whole interval $[0,1]$, yielding:
    \begin{equation*}
        c(\eps) \left( 1 - 2\eps \right)^K
        \leq \prob(E) \leq
        c(\eps) \left( 1 + 2\eps \right)^K \,.
    \end{equation*}
    Since $1-\delta(\eps) \le \prob(E) \leq 1$, we get the following  bounds for $c(\eps)$:
    \begin{equation*}
        \frac{1-\delta(\eps)}{(1+2\eps)^K} \leq c(\eps) \leq \frac{1}{(1-2\eps)^K} \,.
    \end{equation*}
    In particular, this implies that $\lim_{\eps\rightarrow0}c(\eps)=1$.
    Using \eqref{upper_boxes_proba}, we deduce that for any $n\geq n_0$:
    \begin{equation*}
        \prob\left[ \forall k\leq K,\; Y_{(k)}^N\in I_k \right]
        \le \prob\left[ \left\{ \forall k\leq K,\; Y_{(k)}^N\in I_k \right\} \cap E \right] + \delta(\eps)
        \le c(\eps) \prod_{k=1}^K \big( \mu_0(I_k) + 2\eps \big) + \delta(\eps) \,,
    \end{equation*}
    and likewise with \eqref{lower_boxes_proba}:
    \begin{equation*}
        \prob\left[ \forall k\leq K,\; Y_{(k)}^N\in I_k \right] 
        \geq \prob\left[ \left\{ \forall k\leq K,\; Y_{(k)}^N\in I_k \right\} \cap E \right] 
        \geq c(\eps) \prod_{k=1}^K \Big(\big( \mu_0(I_k) - 2\eps \big)_+\Big) \,.
    \end{equation*}
    Both bounds converge to $\mu_0(I_1)\cdots\mu_0(I_K)$ when $\eps\rightarrow0$, proving~\eqref{cv_voulue}. 
    This concludes the proof that \eqref{item:cv_to_mu0} implies \eqref{item:cv_few_first_Y}.

    The fact that \eqref{item:cv_few_first_Y} implies \eqref{item:cv_first_Y} is trivial.
    Let us thus conclude with the proof that \eqref{item:cv_first_Y} implies \eqref{item:cv_to_mu0}.
    By the measure disintegration theorem, there exists a collection of measures $(\tilde\mu_x)_{x \in [0,1]}$ such that for any Borel set $B \subseteq [0,1]^2$, we have
    $\mu(B) = \int_0^1 \tilde\mu_x(B_x) dx$, where $B_x=\{y:(x,y) \in B\}$.
    Informally, if $(X,Y)$ has distribution $\mu$, then $\tilde\mu_x$ is the law of $Y$ knowing that $X=x$.
    By construction, $X_{(1)}^N$ is the smallest element in $\{X_1^N,\dots,X_N^N\}$.
    In particular,
    \[\mathbb P\big[ X_{(1)}^N \ge x \big] =  \mathbb P\Big[ \mathcal P^N \cap ([0,x] \times [0,1]) = \emptyset \Big]
    = \exp\big(-n\mu([0,x] \times [0,1]) \big) = \exp(-nx).\]
    Thus $X_{(1)}^N$ has density $n \exp(-nx)$.
    Conditionally to $X_{(1)}^N=x$, the random variable $Y_{(1)}^N$ has law $\tilde\mu_x$.
    Summing up, for every Borel set $A \subseteq [0,1]$, we have
    \[\mathbb P\big[ Y_{(1)}^N \in A \big] =\mathbb E\Big[\mathbb P\big[ Y_{(1)}^N \in A \,\big|\, X_{(1)}^N \big]\Big] =
    \mathbb E\Big[ \tilde\mu_{X_{(1)}^N}(A) \Big] = \int_0^1 \tilde\mu_x(A) \, n \exp(-nx) dx.\]
    Setting $f_A(x)=\tilde\mu_x(A) \One_{[0,1]}(x)$ and considering its Laplace transform $Lf_A(s):=\int_0^1 e^{-xs} f_A(x)dx$, the above formula writes $\mathbb P\big[ Y_{(1)}^N \in A \big] =n \, Lf_A(n)$.

    Now, Item \eqref{item:cv_first_Y} tells us that, for any continuity set $A$ for $\mu_0$, we have $\lim_{n \to +\infty} \mathbb P\big[ Y_{(1)}^N \in A \big]=\mu_0(A)$, or equivalently $Lf_A(n) \sim \mu_0(A)\, n^{-1}$ for large integers $n$. 
    It is straightforward to extend this estimate to large real numbers $s$, i.e.~we have $Lf_A(s) \sim \mu_0(A)\, s^{-1}$ for large $s$.
    By \cite[Theorem 3 page 445]{Feller}, this implies $\int_0^x f_A(y) dy \sim \mu_0(A) x$ when $x$ tends to $0^+$. 
    But for $x \le 1$, the integral $\int_0^x f_A(y)dy$ is simply $\int_0^x \tilde\mu_y(A)dy=\mu\big([0,x] \times A\big)$.
    Summing up, we get
    \[ \lim_{x \to 0^+} \tfrac1x \mu\big([0,x] \times A\big) = \mu_0(A).\]
    Since this holds for any continuity set $A$ for $\mu_0$, this proves Item \eqref{item:cv_to_mu0}.
\end{proof}

We now \enquote{de-Poissonize} the previous result.
{We use the notation of Section~\ref{ssec:permuton-sample}, namely $(X_i^n,Y_i^n)_{i \le n}$ is an i.i.d.~sample of fixed size $n$ and common distribution $\mu$, and $(X^n_{(i)},Y^n_{(i)})_{i \le n}$ is its reordering with increasing $x$-coordinates.}

\begin{cor}
\label{corol:k-pts-fixed-n}
    The following statements are equivalent:
    \begin{enumerate}[(i)]
        \item \label{item:cv_to_mu0_n} we have the weak convergence of measures
        \begin{equation*}
        \lim_{x\rightarrow0^+} \frac{1}{x} \mu\big([0,x]\times\cdot\big) = \mu_0\,;
        \end{equation*}
    \item \label{item:cv_few_first_Y_n} for any fixed $K \ge 1$, we have the following convergence in distribution in $\mathbb R^k$:
        \begin{equation*}
        \left( Y_{(1)}^n, \dots, Y_{(K)}^n \right) \underset{n\rightarrow\infty}{\longrightarrow}(Y_k)_{1\leq k\leq K}\,,
        \end{equation*}
    where $(Y_k)_{1\leq k\leq K}$ is a sequence of $K$ i.i.d.~random variables distributed according to $\mu_0$.
    \item \label{item:cv_first_Y_n} the random variable $Y_{(1)}^n$ converge in distribution to a random variable $Y_1$ with law $\mu_0$.
    \end{enumerate}
\end{cor}

\begin{proof}
    For each $n\in\bN$, let $\pts^N = \big( (X_1^N,Y_1^N),\dots,(X_N^N,Y_N^N) \big)$ be a Poisson point process of intensity $(n+n^{2/3})\mu$, {listed in a uniform random order}.
    Since the number $N$ of points follows a $\poisson{n+n^{2/3}}$ law, for large $n$, it has fluctuations of order $\sqrt n$ around its mean value $n+n^{2/3}$. 
    In particular, the event $\{N\geq n\}$ happens w.h.p.~as $n$ goes to infinity.
    Conditionally under this event, $\left\{ (X_1^N,Y_1^N),\dots,(X_n^N,Y_n^N) \right\}$ is a family of $n$ i.i.d.~random points distributed under $\mu$.
    We denote by $(X_{(1)}^N,Y_{(1)}^N),\dots,(X_{(N)}^N,Y_{(N)}^N)$ and $(X_{(1)}^n,Y_{(1)}^n),\dots,(X_{(n)}^n,Y_{(n)}^n)$ the reorderings, by increasing $x$-coordinate, of $\left\{ (X_1^N,Y_1^N), \dots, (X_N^N,Y_N^N) \right\}$ and  $\left\{ (X_1^N,Y_1^N), \dots, (X_n^N,Y_n^N) \right\}$ respectively.
    Let $\tau$ be the unique permutation of size $N$ satisfying $X_{(i)}^N = X_{\tau(i)}^N$ for all $1\le i \le N$.
    Since $\tau$ is uniformly random and since $N \le n+2n^{2/3}$ w.h.p., the event $\{ \tau(1)\leq n ,\dots, \tau(K)\leq n \}$ happens w.h.p.~as $n$ goes to infinity.
    {Informally, this event means that the $K$ left-most points of the whole Poisson point process $\pts^N$ belong to the subset of its first $n$ points.}
    Conditionally under this event, one has:
    \begin{equation*}
        \left( \big(X_{(1)}^n,Y_{(1)}^n\big) ,\dots, \big(X_{(K)}^n,Y_{(K)}^n\big) \right) = \left( \big(X_{(1)}^N,Y_{(1)}^N\big) ,\dots, \big(X_{(K)}^N,Y_{(K)}^N\big) \right) .
    \end{equation*}
    \Cref{corol:k-pts-fixed-n} then follows from \Cref{prop:k points}.
\end{proof}

\subsection{Proof of subtree size convergence}

Recall from \Cref{ssec:result-stscv} that for a finite tree $\lT$ and a node $u$ in $\rV$, we denote by $t(\lT,u)$ the proportion of nodes of $\lT$ which are descendants of $u$ (including $u$ itself).
In a binary search tree, this can be computed as follows.

\begin{lemma}\label{lem:subtree-proporitions_in_BSTs}
Let $y_1,\dots,y_n$ be distinct numbers and let $\lT:=\lbst{y_1,\dots,y_n}$ be the corresponding BST.
Let $u$ be a node in $\lT$ and let $k$ be such that $\lT(u)=y_k$.
Then we have
\begin{equation}\label{eq:subtree-proporitions_in_BSTs}
    t(\lT,u) 
    = \frac{1}{|\lT|} \, \Big| \big\{y_k,\dots,y_n\big\} 
    \cap \big( \lTleft(u),\lTright(u) \big) \Big| \,, 
\end{equation}
where $\lTleft$ and $\lTright$ {are extended to finite binary search trees using the definition from} \Cref{ssec:result-stscv}.
\end{lemma}

\begin{proof}
    Consider the iterative construction of $\lbst{y_1,\dots,y_n}$.
    A number $y_i$ in the list is inserted in a node which is a strict descendant of $u$ if
    \begin{itemize}
        \item the node $u$ has been filled before, i.e.~if $i>k$;
        \item the number $y_i$ compares in the same way as $y_k$ to all numbers $\lT(u')$, where $u'$ is an ascendant of $u$. 
        This condition is equivalent to $y_i \in \big( \lTleft(u),\lTright(u) \big)$, see \Cref{fig:right-left-ancestors}.
    \end{itemize}
    Hence, the numerator in \eqref{eq:subtree-proporitions_in_BSTs} is indeed the number of descendants of $u$ in $T$ (including $u$, which corresponds to the label $y_k$). 
    This proves the lemma.
\end{proof}

\begin{proof}[Proof of \Cref{thm:limit}]
    Let us prove the first statement:  
    assume that Assumption~\assumption{2} is satisfied by $\mu$, and let $\mu_0$ be defined by \eqref{eq:def_mu0}.
    
    Let $\lPT{n}{}:=\lbst{\pts^n_\mu}$, skipping the dependency on $\mu$ in the notation.
    Since the subtree size topology is by definition the pointwise convergence of the function $(t(.,u))_{u \in \rV}$, we need to prove the convergence of finite-dimensional distributions.
    Namely, we need to prove that, for any $d \ge 1$ and $u_1,\dots,u_d$ in $\rV$, we have the following convergence in distribution as $n$ tends to $\infty$:
    \begin{equation}
        \label{eq:ToProve-Grubel}
        \big(t( \lPT{n}{} ,u_i) \big)_{i \le d} \longrightarrow \big(\Tlim{\mu}(u_i) \big)_{i \le d}.
    \end{equation}
    As usual, for each $n$, let $(X^n_1,Y^n_1),\dots,(X^n_n,Y^n_n)$ be the i.i.d.~random points in $[0,1]^2$ with common distribution $\mu$ used to construct $ \lbst{\pts^n_\mu}$ and reorder them as a sequence $(X^n_{(1)},Y^n_{(1)}),\dots,(X^n_{(n)},Y^n_{(n)})$ such that $X^n_{(1)} < \dots < X^n_{(n)}$.

    From \Cref{corol:k-pts-fixed-n}, we have the following convergence in distribution for the topology of pointwise convergence:
    \begin{equation}\label{eq:cv_first_Y}
        \big( Y_{(k)}^n \big)_{k \ge 1} 
        \underset{n\rightarrow\infty}{\longrightarrow} 
        (Y_k)_{k \ge 1},
    \end{equation}
    where $Y_1,Y_2,\dots$~is an infinite sequence of i.i.d.~random variables with distribution $\mu_0$.
    Using Skorohod's representation theorem \cite[Section 6]{billingsley1999convergence}, we might assume that the above convergence holds almost surely.

    Since $\mu_0$ has no atoms, the numbers $(Y_k)_{k\ge1}$ are a.s.~all distinct.
    Moreover, the tree $\lbst{Y_1,Y_2,\dots}$ has a.s.~shape $\rV$ (i.e.~there is no empty branch).
    Consequently, a.s., there exists a (random) threshold $K$ such that all nodes $u_i$ belong to $\lbst{Y_1,\dots,Y_K}$.
    Using the convergence \eqref{eq:cv_first_Y}, we know that there exists a (random) threshold $n_0$ such that for $n \ge n_0$, the relative order of $(Y_{(1)}^n,\dots, Y_{(K)}^n)$ is the same as that of $(Y_1,\dots,Y_K)$.
    This implies that the trees $\lT_K^n:=\lbst{Y_{(1)}^n,\dots, Y_{(K)}^n}$ and $\lT_K^\infty:=\lbst{Y_1,\dots,Y_K}$ have the same shape $T_K$.
    Moreover, for any $v$ in $T_K$, the values $\lT_K^n(v)$ and $\lT_K^\infty(v)$  correspond to $Y_{(i)}^n$ and $Y_i$ respectively, for the {\em same} index $i$. 
    Therefore, using again \eqref{eq:cv_first_Y}, we know that $\lT_K^n(v)$ converges to $\lT_K^\infty(v)$ 
    (a.s.~in the probability space created by the application of Skorohod's representation theorem).

    Now, using \Cref{lem:subtree-proporitions_in_BSTs} and the fact that each $u_i$ is filled in $\lT_n$ before step $K=\O_\prob(1)$, we have that
    \[t(\lPT{n}{},u_i) =\frac1n \, \Big| \{ Y^n_1,\dots,Y^n_n\} \cap \big( \lTnleft(u_i),\lTnright(u_i) \big) \Big|  + \o_\prob(1) \,. \]
    Introducing the empirical distribution function of the $(Y^n_i)_{i \le n}$
    \begin{equation}\label{eq:empirical-cdf-proof-STS}
        F_n(y):=\frac1{n} \big| \{ Y^n_1,\dots,Y^n_n\} \cap ( -\infty,y ) \big| \,,
    \end{equation}
    we have 
    \[ t( \lPT{n}{},u_i) = F_n\big( \lTnright(u_i) \big) - F_n\big( \lTnleft(u_i) \big) +\o_\prob(1) \,. \]
    For each fixed $n$, $(Y^n_i)_{1 \le i \le n}$ are i.i.d.~random variables in $[0,1]$.
    Since $\mu$ is a permuton, it satisfies~\eqref{eq:uniform_marginals}, and the common distribution of the $Y^n_i$'s is the uniform distribution.
    {From \Cref{prop:glivenko_cantelli_triangular}, we infer that $F_n$ converges a.s.~uniformly on $[0,1]$ to the identity function
    (the earlier use of Skorohod's representation theorem implies that the $(Y^n_i)_{1 \le i \le n}$ are coupled in a nontrivial way for different values of $n$, but \Cref{prop:glivenko_cantelli_triangular} applies nevertheless).}
    
    Moreover, the above discussion implies that $\lTnright(u_i)$ and $\lTnleft(u_i)$ converge a.s.~to $\lTinfright(u_i)$ and $\lTinfleft(u_i)$ respectively.
    Therefore, a.s.~in the probability space created by the application of Skorohod's representation theorem, we have that, for all $i\le d$,
    \[t( \lT^n,u_i) =\lTinfright(u_i) -\lTinfleft(u_i) +o_\prob(1)=\Tlim{\mu_0}(u_i) +o_\prob(1). \]
    Since a.s.~(joint) convergence implies (joint) convergence in distribution, \eqref{eq:ToProve-Grubel} is proved, concluding the first statement of \Cref{thm:limit}.

    Now let us prove the second statement: 
    we assume that $t( \lPT{n}{} ,u)$ converges in distribution as $n\to\infty$, finitely jointly in $u\in\rV$.
    In particular, the proportion of nodes in the left-subtree of the root, $t( \lPT{n}{} ,0)$, converges in distribution.
    However, using \Cref{prop:glivenko_cantelli_triangular} as before, we have
    \begin{equation*}
        t( \lPT{n}{} ,0) 
        = \frac1n \Big|\Big\{k : Y_k^n < Y_{(1)}^n \Big\}\Big| 
        = F_n\big( Y_{(1)}^n \big)
        = Y_{(1)}^n + o_\prob(1)
    \end{equation*}
    where $F_n$ is again defined by \eqref{eq:empirical-cdf-proof-STS}.
    Therefore, $Y_{(1)}^n$ converges in distribution as $n\to\infty$.
    Using \Cref{corol:k-pts-fixed-n}, \eqref{item:cv_first_Y}$\implies$\eqref{item:cv_to_mu0}, this concludes the proof.
\end{proof}

\section{Some comparison arguments and consequences}

\subsection{Height modification by adding/removing points}

Given a tree $\mathcal T$, a \textit{chain} in $\mathcal T$ is a subset $C$ of its nodes such that for every pair $(v,w)$ in $C$, either $v$ is an ancestor of $w$, or the converse.
We note that the height of $\mathcal T$ is simply the maximal size of a chain in $\mathcal T$, minus 1.
We extend this definition to point sets as follows:
given a set of points $\pts=\{(x_1,y_1),\ldots,(x_n,y_n)\}$ with distinct coordinates, we say that $C\subseteq\pts$ is a chain in $\lbst{\pts}$ if the corresponding nodes in $\lbst{\pts}$ form a chain.

\begin{lemma}\label{lem:ancestor condition}
Let $y=(y_1,\dots,y_n)$ be a list of distinct numbers and $\mathcal T=\lbst{y}$ be the associated BST.
If $i<j$ are two indices then the following are equivalent:
\begin{enumerate}[(i)]
    \item $y_i$ is an ancestor of $y_j$ in $\mathcal T$ (the converse cannot hold);
    \item there is no $k<i$ such that $y_k$ is between $y_i$ and $y_j$, i.e.~such that $(y_i-y_k)(y_j-y_k) < 0$.
\end{enumerate}
\end{lemma}

\begin{proof}
{As in Section~\ref{ssec:proof-overview}, we see $\mathcal T$ as a top tree, formed by the insertion of the first $i-1$ elements, and hanging trees. 
The condition (ii) above exactly stipulates that $y_i$ and $y_j$ are in the same hanging tree in this decomposition.
If this is the case, then $y_i$ is the first vertex inserted in this hanging tree, and therefore is its root, so that $y_i$ is indeed an ancestor of $y_j$. {Conversely}, if $y_i$ and $y_j$ are in different hanging trees, then $y_i$ cannot be an ancestor of $y_j$.}
\end{proof}

\begin{lemma}\label{lem:chain trick}
    Let $\pts_-\subseteq\pts_+$ be two sets of points with distinct $x$- and distinct $y$-coordinates.
    Then, for any chain $C$ of $\lbst{\pts_+}$, the set $C\cap\pts_-$ is a chain of $\lbst{\pts_-}$.
    Consequently, if $\mathcal{C}$ is a chain of maximal size in $\lbst{\pts_+}$, we have
    \begin{align*}
        \height{\pts_-} \ge \height{\pts_+} - \big|\mathcal{C}\,\cap\,(\pts_+\setminus\pts_-)\big|\,.
    \end{align*}
\end{lemma}

\begin{proof}
    Let $y^+=(y_1^+,y_2^+,\dots)$ and $y^-=(y_1^-,y_2^-,\dots)$ denote the sequences of $y$-coordinates of $\pts_+$ and $\pts_-$ read by increasing $x$-coordinate, so that $\lbst{\pts_+} = \lbst{y^+}$ and $\lbst{\pts_-} = \lbst{y^-}$.
    Using \Cref{lem:ancestor condition}, if $C =(y_{i_1}^+,\dots,y_{i_\ell}^+)$ is a chain of $\lbst{y^+}$ then for any triple of indices $k<i_j<i_{j'}$ one has 
    \[ \big( y_{i_j}^+-y_k^+ \big)\big( y^+_{i_{j'}}-y_k^+ \big) > 0 .\]
    Since $y^-\subseteq y^+$, this property still holds when restricted to $y^-$, and therefore $C \cap y^-$ is a chain of $\lbst{y^-}$.
    
    Now if $\mathcal{C}$ is a chain of maximal size in $\lbst{y^+}$, one has $\height{y^+}=|\mathcal{C}|-1$.
    But $\mathcal{C} \cap y^-$ is a chain in $\lbst{y^-}$, implying
    \begin{equation*}
    \height{y^-} \ge \big|\mathcal{C} \cap y^-\big| -1 = |\mathcal{C}| - \big|\mathcal{C} \cap (y^+ \setminus y^-)\big| -1 = \height{y^+} - \big|\mathcal{C} \cap (y^+ \setminus y^-)\big|.\qedhere
    \end{equation*}
\end{proof}

Combining the above lemma with standard thinning properties of Poisson point processes, we get the following useful proposition.

\begin{prop}\label{prop:two fish}
    Let $\rho_-\leq\rho_+$ be two intensity functions defined on the same support $S\subseteq\bR^2$, and $\pts_-,\pts_+$ be two Poisson point processes with intensities $\rho_-$ and $\rho_+$.
    Then, we have
    \begin{align*}
        \height{\pts_-} \succeq
        \binomial{1+\height{\pts_+}}{\inf_{(x,y)\in S} \frac{\rho_-(x,y)}{\rho_+(x,y)}}-1.
    \end{align*}
\end{prop}

\begin{proof}
    Write $r := \inf_{(x,y)\in S} \frac{\rho_-(x,y)}{\rho_+(x,y)}$ where, by convention, $\frac{\rho_-(x,y)}{\rho_+(x,y)}=1$ if $\rho_+(x,y)=0$.
    We couple $\pts_+$ and $\pts_-$ according to the classical thinning process, meaning that $\pts_-$ is constructed by keeping each point $(x,y)$ of $\pts_+$ independently with probability $\rho_-(x,y) / \rho_+(x,y)\geq r$.
    Let $C$ be a set of points of $\pts_+$ corresponding to a chain of maximum length in $\lbst{\pts_+}$ and denote by $K = |C\cap\pts_-| = |C| - \big| C \,\cap\, (\pts_+\setminus\pts_-) \big|$ the number of points on this chain kept by the thinning procedure.
    Using the thinning process, conditionally given $\pts_+$, the following stochastic dominance holds:
    \begin{equation*}
        K \succeq \binomial{|C|}{r}\,.
    \end{equation*}
    Then by Lemma~\ref{lem:chain trick}, since $|C|=\height{\pts_+}+1$, we have
    \begin{equation*}
        \height{\pts_-}
        \ge \height{\pts_+} -
        \big|C\,\cap\,(\pts_+\setminus\pts_-)\big|=K-1\,.\qedhere
    \end{equation*}
\end{proof}

\begin{remark}\label{rem:height and points}
Considering as in \Cref{lem:chain trick} two sets of points $\pts_-\subseteq\pts_+$ with distinct $x$- and $y$-coordinates, the height $\height{\pts_-}$ can be much bigger than $\height{\pts_+}$, even while removing a single point.
To see this, take $n$ odd, let $\pts_-=\{(i/n,i/n), 0 < i < n\}$, and $\pts_+=\pts_- \cup \{(0,1/2)\}$.
The points in $\pts_-$ form an increasing sequence, so that $\lbst{\pts_-}$ consists in a single branch growing to the right, and has height $n-2$. 
On the other hand, the root in $\lbst{\pts_+}$ has label $1/2$ and divides the tree into two equal parts of size $(n-1)/2$. 
Each of this part has height $(n-1)/2-1$, so that $\height{\pts_+} = (n-1)/2$.
\end{remark}

\subsection{A de-Poissonization result}

The previous comparison lemma can also be used to de-Poissonize convergence results for the height of BSTs.
{Recall that $\pts^n_\mu$ and $\pts^N_\mu$ denote respectively a sample of $n$ i.i.d.~random points with law $\mu$, and a Poisson point process with intensity $n \mu$.}

\begin{thm}\label{thm:Poisson fixed}
    Let $\mu$ be a permuton. 
    Let $f:\bN \rightarrow [1,\infty)$ be a function such that there exists $\frac12 < \alpha <1$ satisfying
    \begin{align}\label{eq:regularity_f}
        \sup_{|\delta|\leq n^\alpha} \left| \frac{f(n+\delta)}{f(n)} - 1 \right|
        \underset{n\to\infty}{\longrightarrow} 0 \,.
    \end{align}
    Then we have
    \begin{align*}
        \frac{\height{\pts^n_\mu}}{f(n)} \overset{\prob}{\longrightarrow}1~\Longleftrightarrow~\frac{\height{\pts^N_\mu}}{f(n)} \overset{\prob}{\longrightarrow}1
    \end{align*}
    as $n\to\infty$.
    Moreover all powers of $\frac{\height{\pts^n_\mu}}{f(n)}$ are uniformly integrable if and only if all powers of $\frac{\height{\pts^N_\mu}}{f(n)}$ are uniformly integrable.
\end{thm}

\begin{proof}
    First suppose that $\frac{\height{\pts_\mu^n}}{f(n)} \overset{\prob}{\longrightarrow}1$.
    Then $\frac{\height{\pts_\mu^N}}{f(N)} \overset{\prob}{\longrightarrow}1$ as $n\to\infty$.
    Moreover, since $\alpha >1/2$, w.h.p.~it holds that $n- n^\alpha\leq N\leq n+ n^\alpha$, so the regularity hypothesis \eqref{eq:regularity_f} on $f$ implies $f(n)/f(N) \overset{\prob}{\longrightarrow} 1$ and we conclude that $\frac{\height{\pts_\mu^N}}{f(n)} \overset{\prob}{\longrightarrow}1$.
    
    Now suppose that all powers of $\frac{\height{\pts_\mu^n}}{f(n)}$ are uniformly integrable, or equivalently that for every integer $p$, the sequence $\expec\left[ \frac{\height{\pts_\mu^n}^p}{f(n)^p} \right]$ is bounded in $n$.
    This implies that the sequence $\expec\left[ \frac{ \height{\pts_\mu^N}^p }{f(N)^p} \right]$ is also bounded in $n$.
    Therefore:
    \begin{align*}
        \expec\left[ \frac{ \height{\pts_\mu^N}^p }{f(n)^p} \right]
        &\leq \expec\left[ \frac{ \height{\pts_\mu^N}^p }{f(n)^p} \mathbf1_{|N-n|\le n^\alpha} \right] \, + \, \expec\left[ N^p \mathbf1_{|N-n|> n^\alpha} \right]
        \\&\leq \expec\left[ \frac{ \height{\pts_\mu^N}^p }{f(N)^p} \right] \sup_{|\delta|\leq n^\alpha}\frac{f(n+\delta)^p}{f(n)^p}
\, +\,  \sqrt{\expec\left[ N^{2p} \right] \prob\left( |N-n|>n^\alpha \right)}.
    \end{align*}
    The first term is bounded by assumption, while the second one is easily proved to tend to $0$, using \Cref{lem:Poisson_bounds} and \Cref{eq:moment_poisson}.
    We conclude that, for any $p$, the quantity $\expec\left[ \frac{\height{\pts_\mu^n}^p}{f(n)^p} \right]$ is bounded in $n$, and thus, all powers of $\frac{\height{\pts_\mu^n}}{f(n)}$ are uniformly integrable.
    \bigskip
    
    The converse implications require de-Poissonization and are more subtle.
    The global idea is to bound the height of $\lbst{\pts_\mu^n}$ by two Poissonized versions, using \Cref{lem:chain trick}.
    However, making this idea work (both for probability convergence and for $L^p$ boundedness) requires some computation.
    
    Let $N_+\sim\poisson{n+ n^\alpha}$ and $\pts^{N_+}$ be a Poisson point process of intensity $(n+ n^\alpha)\mu$.
    Since $\alpha>\frac12$, the event $E_+ = \{N_+\geq n\}$ holds w.h.p.~as $n\to\infty$.
    Under $E_+$, define $\pts^n$ as a uniform subset of size $n$ in $\pts^{N_+}$.
    Then $\pts^n$ is a set of $n$ i.i.d.~points distributed under $\mu$.
    Similarly, let $N_-\sim\poisson{n- n^\alpha}$ and note that the event $E_-=\{N_-\leq n\}$ holds w.h.p.~as $n\to\infty$.
    Under the event $E_-$ and conditionally given $N_-$, define $\pts^{N_-}$ as a uniform subset of size $N_-$ in $\pts^{n}$.
    Then conditionally under $E_-$, the set $\pts^{N_-}$ is distributed like a Poisson point process of intensity $(n- n^\alpha)\mu$ conditioned to have at most $n$ points.
    By applying \Cref{lem:chain trick} to both $\pts^n\subseteq\pts^{N_+}$ and $\pts^{N_-}\subseteq\pts^n$, conditionally under $E=E_-\cap E_+$, we have that
    \begin{align}\label{eq:encadrement fixed poisson}
        \height{\pts^{N_+}} - \big|C_+\,\cap\,(\pts^{N_+}\setminus\pts^n)\big|
        \le \height{\pts^n} \le 
        \height{\pts^{N_-}} + \big|C\,\cap\,(\pts^{n}\setminus\pts^{N_-})\big|
    \end{align}
    where $C_+,C$ are arbitrary chains of maximal size in $\lbst{\pts^{N_+}}$ and $\lbst{\pts^{n}}$ respectively.
    Under $E$ and conditionally on $N_+$ and $\height{\pts^{N_+}}$, we may apply \Cref{lem:size dom} to obtain:
    \begin{align}\label{eq:dom+}
        \big|C_+\,\cap\,(\pts^{N_+}\setminus\pts^n)\big|
        \preceq \binomial{1+\height{\pts^{N_+}}}{\frac{N_+-n}{N_+ - \height{\pts^{N_+}} - 1}} \,.
    \end{align}
    Similarly, assuming $E$ and conditionally given $N_-$ and $\height{\pts^n}$, it holds that:
    \begin{align}\label{eq:dom-}
        \big|C\,\cap\,(\pts^{n}\setminus\pts^{N_-})\big|
        \preceq \binomial{1+\height{\pts^{n}}}{\frac{n-N_-}{n-\height{\pts^{n}}-1}} \,.
    \end{align}
    
    {We now} suppose that $\frac{\height{\pts_\mu^N}}{f(n)} \overset{\prob}{\longrightarrow}1$ as $n$ goes to infinity, {and we want to prove that $\frac{\height{\pts_\mu^n}}{f(n)}$ tends to $1$ in probability as well}.
    Using the regularity hypothesis \eqref{eq:regularity_f} on $f$, we have $\frac{\height{\pts^{N_{\pm}}}}{f(n)} \overset{\prob}{\longrightarrow}1$.
    {Throughout the proof, we fix} constants $\beta \in(\alpha,1)$ and $\delta>0$.
    \medskip
 
    \underline{Lower bound on $\height{\pts_\mu^n}$}.
    Using \eqref{eq:encadrement fixed poisson}, as $n$ goes to infinity:
    \begin{align*}
        &\prob\left[ \height{\pts_\mu^n} \le (1-\delta)f(n) \right]
        \\&\leq \prob\left[ \height{\pts_\mu^n} \le (1-\delta)f(n) \;\wedge\; E \right] 
        + \prob[E^c]
        \\&\leq \prob\left[\Big. \height{\pts^{N_+}} - \Big|C_+\,\cap\,(\pts^{N_+}\setminus\pts^n)\Big| \le (1-\delta)f(n) \right] 
        + o(1)
        \\&\leq \prob\left[\Big. \height{\pts^{N_+}} \le (1-\delta/2)f(n) \right] 
        + \prob\left[ \Big|C_+\,\cap\,(\pts^{N_+}\setminus\pts^n)\Big| \ge \frac\delta2 f(n) \right]
        + o(1) \,.
    \end{align*}
    In the last line, we can conclude directly from the convergence $\frac{\height{\pts^{N_+}}}{f(n)} \overset{\prob}{\longrightarrow}1$ that the first probability is a $o(1)$, but the second one requires more attention.
    
    Notice that under the event $\left| C_+\,\cap\, \left( \pts^{N_+}\setminus\pts^n \right)\right| \geq \frac\delta2 f(n)$, we have $f(n)\leq \frac 2\delta (N_+-n)$. 
    Moreover, w.h.p.~as $n\to\infty$, it holds that $N_+-n\leq n^\beta$ and $1+\height{\pts_\mu^{N_+}} \le 2f(n)$. 
    Thus the parameters of the binomial random variable in \eqref{eq:dom+} are bounded w.h.p.~by $2f(n)$ and $2n^{\beta-1}$ respectively (recall that $\beta<1$).
    Denote by $S_n$ a $\binomial{\lfloor2f(n)\rfloor}{2n^{\beta-1}}$ random variable to obtain the following:
    \begin{align*}
        \prob\left[ \left| C_+\,\cap\, \left(\pts^{N_+}\setminus\pts^n \right)\right| \ge \frac\delta2 f(n) \right]
        \le \prob\left[ S_n \ge \frac\delta2 f(n) \right] + o(1)
        {\leq\frac{2\, \expec[S_n]}{\delta f(n)}+o(1)} 
        = o(1) 
    \end{align*}
    This concludes the proof that
    \begin{align*}
        \prob\left[ \height{\pts_\mu^n} \le (1-\delta)f(n) \right]
        \underset{n\to\infty}{\longrightarrow}0 \,.
    \end{align*}\medskip

    \underline{Upper bound on $\height{\pts_\mu^n}$}.
    For each $n\in\bN$ we distinguish between two cases.
    \begin{itemize}
        \item Suppose $f(n)\ge n^\beta$. 
        Then we use \eqref{eq:encadrement fixed poisson} as before:
        \begin{align*}
            &\prob\left[ \height{\pts_\mu^n} \ge (1+\delta)f(n) \right]
            \\&\le \prob\left[\Big. \height{\pts^{N_-}} +\big|C\,\cap\,(\pts^{n}\setminus\pts^{N_-})\big| \ge (1+\delta)f(n) \right]
            + o(1)
            \\&\le \prob\left[ \height{\pts^{N_-}} \ge (1+\delta/2)f(n) \right] 
            + \prob\left[ n-N_- \ge \delta n^\beta /2 \right]
            + o(1)
            \\&=o(1) \,.
        \end{align*}
        \item Suppose $f(n)\le n^\beta$. 
        From \eqref{eq:encadrement fixed poisson} and the trivial bound $\big|C\,\cap\,(\pts^{n}\setminus\pts^{N_-})\big| \leq n-N_-$, we have $n- \height{\pts^{n}} \ge N_- - \height{\pts^{N_-}}$.
        But, w.h.p., $N_- \ge n - n^\beta$ and $\height{\pts^{N_-}} \le 2 f(n) \le 2 n^\beta$, implying that $n - \height{\pts^{n}} \ge n-3n^\beta$.
        Using again that $n-N^-\le n^\beta$ w.h.p., the probability parameter of the binomial random variable in \eqref{eq:dom-} tends to $0$ in probability.
        Hence w.h.p., one has $\big|C\,\cap\,(\pts^{n}\setminus\pts^{N_-})\big| \le \frac{\delta/2}{1+\delta} \height{\pts^{n}}$
        (recall that $\delta$ is an arbitrary positive constant, and thus, so is $\frac{\delta/2}{1+\delta}$).
        The upper bound of \eqref{eq:encadrement fixed poisson} yields w.h.p.:
        \begin{align*}
            \left( 1-\frac{\delta/2}{1+\delta} \right) \height{\pts^n} \le \height{\pts^{N_-}} \,,
        \end{align*}
        and then
        \begin{align*}
            &\prob\left[ \height{\pts_\mu^n} \ge (1+\delta)f(n) \right]
            \\&= \prob\left[ \left( 1-\frac{\delta/2}{1+\delta} \right) \height{\pts_\mu^n} \geq (1+\delta/2)f(n) \right]
            \\&\leq \prob\left[ \height{\pts^{N_-}} \ge (1+\delta/2)f(n) \right] + o(1)
            \\&= o(1) \,.
        \end{align*}
    \end{itemize}
    We have thus proved that $\frac{\height{\pts_\mu^n}}{f(n)} \overset{\prob}{\longrightarrow}1$.
    \medskip

 \underline{Uniform integrability.}   
    Finally suppose that for every integer $p$, the sequence $\expec\left[ \frac{\height{\pts_\mu^N}^p}{f(n)^p} \right]$ is bounded in $n$. 
    Thanks to our hypothesis on $f$, the sequence $\expec\left[ \frac{\height{\pts^{N_-}}^p}{f(n)^p} \mathbf{1}_E \right]$ is also bounded in $n$.
    Then \eqref{eq:encadrement fixed poisson}, together with the convexity inequality
    $(a+b)^p \le 2^{p-1}(a^p + b^p)$ for $a,b \ge 0$ and $p \ge 1$, yields:
    \begin{align}\label{eq:UI bound fixed Poisson}
        \expec\left[ \frac{\height{\pts_\mu^{n}}^p}{f(n)^p} \right]
        \leq n^p\cdot\prob(E^c)
        + 2^{p-1} \expec\left[ \frac{\height{\pts^{N_-}}^p}{f(n)^p} \mathbf{1}_E \right]
        + 2^{p-1} \expec\left[ \frac{\big|C\,\cap\,(\pts^{n}\setminus\pts^{N_-})\big|^p}{f(n)^p} \mathbf{1}_E \right].
    \end{align}
    But $n^p\cdot\prob(E^c)$ converges to $0$ (as a consequence of \Cref{lem:Poisson_bounds}, the probability actually decreases at least as fast as $e^{-n^{2\alpha-1}}$), while the second term was already identified as being bounded in $n$. 
    Let us consider the last term, which we split as follows
     \begin{multline}\label{eq:C^p-split}
        \expec\left[ \frac{\big|C\,\cap\,(\pts^{n}\setminus\pts^{N_-})\big|^p}{f(n)^p} \mathbf{1}_E \right] \\
        = \expec\left[ \frac{\big|C\,\cap\,(\pts^{n}\setminus\pts^{N_-})\big|^p}{f(n)^p} \mathbf{1}_E \mathbf{1}_{\height{\pts^{N_-}} \ge n^\beta}\right]
        + \expec\left[ \frac{\big|C\,\cap\,(\pts^{n}\setminus\pts^{N_-})\big|^p}{f(n)^p} \mathbf{1}_E \mathbf{1}_{\height{\pts^{N_-}} < n^\beta}\right]
        \end{multline}
        For the first term, we bound $\big|C\,\cap\,(\pts^{n}\setminus\pts^{N_-})\big|$ by $(n-N_-)$ and, using the Cauchy--Schwarz inequality, we get
        \begin{align}\label{eq:C^p-first-term}
        \expec\left[ \frac{\big|C\,\cap\,(\pts^{n}\setminus\pts^{N_-})\big|^p}{f(n)^p} \mathbf{1}_E \mathbf{1}_{\height{\pts^{N_-}} \ge n^\beta}\right]&\le
        \expec\left[ \frac{(n-N_-)^p}{n^{\beta p}} \frac{\height{\pts^{N_-}}^p}{f(n)^p} \mathbf{1}_E \right] \\ &\le
         \sqrt{ \expec\left[ \frac{(n-N_-)^{2p}}{n^{2\beta p}} \right] \expec\left[ \frac{\height{\pts^{N_-}}^{2p}}{f(n)^{2p}} \mathbf{1}_E \right] }.\notag
        \end{align}
    Since $n-N_-=n^\alpha-(N_- - \expec[N_-])$ and since it is well known that $\frac1{\sqrt{n-n^\alpha}} (N_- - \expec[N_-])$ converges in distribution and in moments to a standard Gaussian random variable, the first expectation under the square root tends to $0$ as $n$ tends to $\infty$ (recall that $\beta>\alpha>\frac12$). 
    Also, it has been observed above that the second expectation under the square root is bounded in $n$. 
    Thus the left-hand side of \eqref{eq:C^p-first-term} tends to $0$ as $n$ tends to $\infty$.
    
    We now consider the second term in \eqref{eq:C^p-split}.
    From \Cref{lem:Poisson_bounds}, one has $N_- \ge n - n^\beta$ outside a set of exponentially small probability. 
    When this holds together with the events $E$ and $\height{\pts^{N_-}}< n^\beta$, one has, for $n$ large enough,
   \[ \frac{n-N_-}{n-\height{\pts^{n}}-1} \le  \frac{n^\beta}{n-n^\beta-1} \le 2n^{\beta-1}.\]
    Hence, using \eqref{eq:dom-}, and writing $S_n$ for a $\binomial{1+\height{\pts_\mu^n}}{2n^{\beta-1}}$ random variable, we get:
    \begin{align*}
    \expec\left[ \frac{\big|C\,\cap\,(\pts^{n}\setminus\pts^{N_-})\big|^p}{f(n)^p} \mathbf{1}_E \mathbf{1}_{\height{\pts^{N_-}} < n^\beta}\right]
    \le \expec\left[ \frac{{S}_n^p}{f(n)^p} \right]+ n^p \, \prob[N_- < n - n^\beta] 
    =\expec\left[ \frac{{S}_n^p}{f(n)^p} \right]+ o(1).
    \end{align*}
    The moments of a $\binomial{M}{q}$ random variable $X$ are easily bounded by $\expec[X^p] \le M^p \, q$, implying that
    \[\expec\left[ \frac{\big|C\,\cap\,(\pts^{n}\setminus\pts^{N_-})\big|^p}{f(n)^p} \mathbf{1}_E \mathbf{1}_{\height{\pts^{N_-}}< n^\beta}\right]
    \le  \expec\left[ \frac{\left( 1+\height{\pts_\mu^n} \right)^p}{f(n)^p} \right] \, 2n^{\beta-1} +o(1).\]
    Using this last estimate, \eqref{eq:UI bound fixed Poisson}-\eqref{eq:C^p-split} and the fact that \eqref{eq:C^p-first-term} tends to $0$, we get that
    \[  \expec\left[ \frac{\height{\pts_\mu^{n}}^p}{f(n)^p} \right]
    \leq \O(1) + 2n^{\beta-1} \expec\left[ \frac{\left( 1+\height{\pts_\mu^{n}} \right)^p}{f(n)^p} \right]. \]
Using again that $(a+b)^p\leq2^{p-1}(a^p+b^p)$ and since $2n^{\beta-1}$ tends to $0$, this implies that $\expec\left[ \frac{\height{\pts_\mu^{n}}^p}{f(n)^p} \right]$ is bounded (for all $p$), proving the uniform integrability of all powers of $\frac{\height{\pts_\mu^{n}}}{f(n)}$.
\end{proof}

\subsection{A connection with monotone subsequences and extreme deviation bounds}

We start by recalling standard definitions.
Let $\sigma$ be a permutation of $\{1,\dots,n\}$.
An increasing subsequence of $\sigma$ is a sequence of indices $i_1<\dots<i_k$ such that $\sigma(i_1)<\dots<\sigma(i_k)$.
The maximum length of an increasing subsequence of $\sigma$ is then denoted by $\LIS{\sigma}$.
We define similarly $\LDS{\sigma}$, the maximum length of a decreasing subsequence of $\sigma$.

\begin{lemma}\label{lem: upper bound LIS LDS}
    Let $\sigma$ be a permutation of $\{1,\dots,n\}$.
    Then
    \begin{equation*}
        \height{\sigma} \le \LIS{\sigma} + \LDS{\sigma} .
    \end{equation*}
\end{lemma}

\begin{proof}
    Let $i_1<\dots<i_k$ be a sequence of integers such that $\sigma(i_1),\dots,\sigma(i_k)$ label nodes on a chain $C$ of $\lbst{\sigma}$.
    Define $\mathcal{I_R}$ (resp.~$\mathcal{I_L}$) as the family of $i_j$'s such that the node following $\sigma(i_j)$ in $C$ lies in its right subtree (resp.~left subtree).
    By construction of the BST, $\mathcal{I_R} \cup \{i_k\}$ and $\mathcal{I_L} \cup \{i_k\}$ form respectively an increasing and a decreasing subsequence of $\sigma$.
    The lemma follows.
\end{proof}

Combining this lemma with \cite[Proposition 3.2]{borga2022permutons}, we get that for any integrable function $\rho$, the quantity $\frac{1}{n}\height{\pts_\rho^N}$ tends to $0$ in probability, as $n\to\infty$.
We will need a more quantitative version of this, valid only for bounded functions $\rho$.
We start with an extreme deviation bound\footnote{
We use the term ``extreme devitation'' since, for the longest increasing subsequence, the usual ``large deviation framework'' consists in studying $\prob\left[ \LIS{\sigma^n} \ge x \sqrt n \right]$ for $x >2$, see \cite{seepalainen1998deviations}.
We look here at much rarer events.} for the longest monotone subsequences.
\begin{lemma}\label{lem: extreme deviation LIS}
  For each integer $n$, let $\sigma^{n}$ be a uniform permutation of $\{1,\dots,n\}$.
  Then we have
    \begin{equation*}
        \prob\left[ \LIS{\sigma^n} \ge \frac{n}{\log n} \right] \le \exp\left( -n + o(n) \right) .
    \end{equation*}
\end{lemma}

\begin{proof}
    This is a straightforward application of the first moment method.
    Write, with $k=\lfloor\frac{n}{\log n}\rfloor$:
    \begin{align*}
        \prob\left[ \LIS{\sigma^n} \ge \frac{n}{\log n} \right]
        &\le \expec\left[ \text{number of increasing subsequences of length } k \text{ in } \sigma^n\right]=\frac{1}{k!}\binom{n}{k}
    \end{align*}
    Now, using Stirling's formula along with $k =\lfloor\frac{n}{\log n}\rfloor = o(n)$, we obtain
    \begin{align*}
        \frac{1}{k!}\binom{n}{k}=\frac{n!}{k!^2(n-k)!}
        &= e^{n\log n - n - 2k\log k -(n-k)\log(n-k) + (n-k) + o(n)} 
        = e^{-n + o(n)} \,.\qedhere
    \end{align*}
\end{proof}

\begin{cor}\label{cor: extreme deviation BST height}
    For any $M>0$ and $\eps>0$, there exists $n_0=n_0(M,\eps)$ such that the following holds.
    For any $0<\zeta\le1$, any function $\rho:[0,1]^2\to[0,\infty)$ bounded by $M$ and supported on some rectangle $[a,b]\times[c,d]$ with $(b-a)(d-c) \le \zeta$, and for any integer $n > n_0/\zeta$:
    \begin{equation*}
        \prob\left[ \height{\pts_\rho^N} > 2\eps \zeta n \right] \le 4\exp\left( -\frac\eps2 \zeta n \log(\zeta n) \right) .
    \end{equation*}
\end{cor}

\begin{proof}
    By increasing the size of the rectangle $[a,b]\times[c,d]$, we can assume without loss of generality that $(b-a)(d-c) = \zeta$.
    Now let us stretch $\rho$ from $[a,b]\times[c,d]$ onto $[0,1]^2$.
    Define 
    \[ g : (x,y)\in [a,b]\times[c,d] \mapsto \left( \frac{x-a}{b-a} , \frac{y-c}{d-c} \right)\in [0,1]^2 \,. \]
    Then $\widetilde{\pts} = g\left( \pts_\rho^N \right)$ is a Poisson point process with intensity $n \zeta \rho\circ g^{-1}$ on $[0,1]^2$.
    Moreover, the transformation $g$ does not change the relative order of points, so $\lbst{\widetilde{\pts}}$ and $\lbst{\pts_\rho^N}$ {have the same shape}.
    From now on we work with $\widetilde\pts$.
      
    The density $\rho\circ g^{-1}$ is bounded above by $M$ on $[0,1]^2$.
    Thus, by a standard thickening procedure, we can construct a Poisson point process $\widehat\pts$ with constant intensity $n\zeta M$ such that a.s.~$\widetilde\pts \subset \widehat\pts$.
    Let $\widetilde\sigma$ and $\widehat\sigma$ be the permutations induced by $\widetilde\pts$ and $\widehat\pts$, respectively.
    Using \Cref{lem: upper bound LIS LDS}, a.s.~it holds that:
    \begin{equation}\label{eq: using monotonicity of LIS}
      \height{\widetilde\pts} \le \LIS{\widetilde\sigma} + \LDS{\widetilde\sigma}
      \le \LIS{\widehat\sigma} + \LDS{\widehat\sigma} .
    \end{equation}
    The permutation $\widehat\sigma$ has a random size $\widehat N$, which follows a $\poisson{n\zeta M}$ law.
    Moreover, conditionally to its size, it is uniformly distributed. 
    By \Cref{lem:Poisson_bounds}, for any $n$ large enough such that $\eps\log(\zeta n)\geq M$, we get:
    \begin{equation*}
      \prob\left[\widehat N > \eps \zeta n \log(\zeta n)\right]
      \le e^{-n\zeta M}\left(\frac{eM}{\eps \log(\zeta n)}\right)^{\eps\zeta n\log(\zeta n)}
      \le \exp\left(\eps\zeta n\log(\zeta n)\log\left(\frac{eM}{\eps \log(\zeta n)}\right)\right) .
    \end{equation*}
    For large enough {$\zeta n$ (with a threshold depending on $M$ and $\eps$)}, we have $\log\left(\frac{eM}{\eps \log(\zeta n)}\right)\le-1$, and thus
    \begin{equation}\label{eq:deviation_Nhat}
      \prob\left[\widehat N > \eps \zeta n \log(\zeta n)\right] \le \exp\left(-\eps\zeta n\log(\zeta n)\right)
    \end{equation}
    Define $n'= \lfloor \eps \zeta n\log(\zeta n)\rfloor$ and write $\sigma^{n'}$ for a uniform permutation of $\{1,\dots,n'\}$.
    As $\frac{n'}{\log(n')} \le \eps \zeta n$ for large enough $\zeta n$, we have
    \[
        \prob\left[ \LIS{\widehat\sigma} > \eps\zeta n \right]
        \le \prob\left[ \LIS{\sigma^{n'}} > \frac{n'}{\log(n')} \right] + \prob\left[\widehat N > \eps \zeta n \log(\zeta n)\right] \,.
    \]
    {Using \Cref{lem: extreme deviation LIS} and \Cref{eq:deviation_Nhat}, we get that, for large enough $\zeta n$,}
    \[
        \prob\left[ \LIS{\widehat\sigma} > \eps\zeta n \right]
        \le e^{-n'+o(n')} + e^{-\eps\zeta n\log(\zeta n)}
        \le 2e^{-\frac12 \eps\zeta n\log(\zeta n)} \,.
    \]
    The same holds for $\LDS{\widehat\sigma}$, and {\Cref{eq: using monotonicity of LIS} allows us to conclude the proof of the corollary}.
\end{proof}

\section{Height of BSTs of permuton samples}
\label{sec:height-proof}

Before starting the proof, let us introduce some notation related to the decomposition presented in the introduction (Section~\ref{ssec:proof-overview}).
Let $\pts=\{(x_1,y_1),\ldots,(x_N,y_N)\}$ {be a set of points} in $(0,1)^2$ with distinct $x$- and distinct $y$-coordinates.
For any $\beta\in(0,1)$, write $\pts(\beta)=\pts\cap ([0,\beta]\times[0,1])$ for the set of points in a band of width $\beta$ on the left.
Now write $K_\beta=|\pts(\beta)|$ and let $y_{(1)}<\dots<y_{(K_\beta)}$ be the ordered $y$-coordinates of the points in $\pts(\beta)$.
Then for each integer $0\leq k\leq K_\beta$, define $I_k = (y_{(k)},y_{(k+1)})$ where we used the convention $y_{(0)}=0$ and $y_{(K_\beta+1)}=1$.
In words, $I_0,\dots,I_{K_\beta}$ are the gaps between the points $\{0,y_{(1)},\dots,y_{(K_\beta)},1\}$, enumerated from lowest to highest.
Finally, define
\begin{align*}
    \pts_k(\beta):=\pts\cap \big( \, (\beta,1]\times I_k \big).
\end{align*}

Recall that $\lbst{\pts(\beta)}$ and $\left(\big. \lbst{\pts_k(\beta)} \right)_{0 \le k \le K_\beta}$ are respectively called the \textit{top tree} and the \textit{hanging trees} of $\lbst{\pts}$.
One can see that the top and hanging trees are indeed subtrees of $\lbst{\pts(\beta)}$.
The entire tree can then be reconstructed as follows: 
start with $\lbst{\pts(\beta)}$, and for each $0<k<K_\beta$ do the following.
Write $v_k$, resp.~$v_{k+1}$, for the node labeled $y_{(k)}$, resp.~$y_{(k+1)}$, in $\lbst{\pts(\beta)}$ and notice that necessarily one is an ancestor of the other.
If $v_k$ is deeper than $v_{k+1}$ then attach $\lbst{\pts_k(\beta)}$ to the right of $v_k$, otherwise attach it to the left of $v_{k+1}$.
Finally, attach $\lbst{\pts_0(\beta)}$ to the left of the node labeled $y_{(1)}$ and $\lbst{\pts_{K_\beta}(\beta)}$ to the right of the node labeled $y_{(K_\beta)}$.
The reader can go back to \Cref{fig:bands} for an illustration.
This construction yields the following lemma:
\begin{lemma}\label{lem:height bound}
    Let $\pts$ be a point set of $[0,1]^2$ with distinct $x$- and distinct $y$-coordinates.
    Then for any $\beta\in(0,1)$:
    \begin{align*}
        \height{\pts(\beta)} \le \height{\pts} \le \height{\pts(\beta)} +1+ \max_{0\le k\le K_\beta}\Big\{\height{\pts_k(\beta)}\Big\}\,.
    \end{align*}
\end{lemma}

\subsection{Controlling the height of the top tree}

\begin{prop}\label{prop:control-top-tree}
    Let $R=[x_1,x_2]\times[y_1,y_2]$ be a rectangle with non-empty interior and $\rho: R \to(0,\infty)$ be a continuous, positive intensity function.
    For each integer $n$, let $\pts^N_\rho$ be a Poisson point process with intensity $n\rho$.
    Let $m\leq M$ be positive real numbers such that $m\leq\rho\leq M$ holds a.e.~on $R$ and write
    \begin{align*}
        \eta=\frac{M-m}{m}\,.
    \end{align*}
    Then for any $\eps>0$, we have
    \begin{align*}
        \lim_{n\to\infty}
        \prob\left[ \left|\frac{\height{\pts_\rho^N}}{c^*\log n}-1\right| > \eta+\eps \right] =0\,.
    \end{align*}
    Moreover, for any $p>0$, the sequence of random variables $\frac{\height{\pts_\rho^N}^p}{\log(n)^p}$ is uniformly integrable.
\end{prop}

\begin{proof}
    Write $\zeta=(x_2-x_1)(y_2-y_1) >0$ for the area of $R$.
    Note that
    \begin{align*}
        \frac{m}{M}=1+\frac{m-M}{M}\geq1-\frac{M-m}{m}=1-\eta
        \qquad;\qquad
        \frac{M}{m}=1+\frac{M-m}{m}=1+\eta\,.
    \end{align*}
    Using Proposition~\ref{prop:two fish} with $\rho_-=n\rho$ and $\rho_+=nM$ on $R$, we obtain
    \begin{align*}
        \height{\pts_\rho^N}
        \succeq
        \binomial{1+\height{\pts_+}}{\frac{m}{M}}-1,
    \end{align*}
    where $\pts_+ := \pts_{\rho_+}^N$. 
    Since $\rho_+$ is a constant density, the tree $\lbst{\pts_+}$ is the BST of a uniform permutation of random size $\poisson{n\zeta M}$.
    According to \cite[Theorem 5.1]{devroye1986note}, $\height{\pts_+}$ then behaves as $c^*\log(|\pts_+|)$ as $n\rightarrow\infty$ in probability, which leads to
    \begin{align*}
        1+\height{\pts_+}
        &= c^*\log(n\zeta M)+o_\prob(\log n)=c^*\log n+o_\prob(\log n) \,.
    \end{align*}
    Using that a $\binomial{a\log n}{m/M}$ random variable is concentrated around its mean $(am/M)\log n$, we get:
    \begin{align*}
        \height{\pts_\rho^N}
        \geq \frac{m}{M}
        \big( c^*\log n - o_\prob(\log n) \big)
        \geq \big(1-\eta-o_\prob(1)\big)c^*\log n\,.
    \end{align*}
    Similarly, using Proposition~\ref{prop:two fish} with $\rho_-=nm$ and $\rho_+=n\rho$ we obtain
    \begin{align}\label{eq:domin top bst binom}
        \height{\pts_-}
        \succeq
        \binomial{1+\height{\pts_\rho^N}}{\frac{m}{M}}-1,
    \end{align}
    where $\pts_- := \pts_{\rho_-}^N$. 
    As before, we observe that $\lbst{\pts_-}$ is the BST of a uniform permutation of random size $\poisson{n\zeta m}$.
    This implies
    \begin{align*}
        \height{\pts_\rho^N}
        \leq \frac{M}{m}c^*\log(n\zeta m)+o_\prob(\log n)=\big(1+\eta+o_\prob(1)\big) c^*\log n \,,
    \end{align*}
    which concludes the proof of the first claim.

    For the uniform integrability claim, let us fix $p>0$ and establish boundedness of $\expec\left[ \frac{\height{\pts_\rho^N}^p}{\log(n)^p} \right]$ in $n$.
    Conditionally given $\height{\pts_\rho^N}$, write $S_n+1$ for a $\binomial{1+\height{\pts_\rho^N}}{\frac{m}{M}}$ random variable.
    Then, using Hoeffding's inequality:
    \begin{align*}
        \prob\left[ \left. S_n < \frac{m}{2M}\left(1+\height{\pts_\rho^N}\right)-1 \;\right|\, \height{\pts_\rho^N} \right] \leq e^{-\frac{m^2}{2M^2}\left(1+\height{\pts_\rho^N}\right)}
    \end{align*}
    and therefore, for any $n\geq e$:
    \begin{align*}
        &\expec\left[ \frac{\height{\pts_\rho^N}^p}{\log(n)^p} \right]
        \\&\leq \expec\left[ \frac{\height{\pts_\rho^N}^p}{\log(n)^p} \mathbf{1}_{S_n < \frac{m}{2M}\left(1+\height{\pts_\rho^N}\right)-1} \right]
        + \expec\left[ \frac{\height{\pts_\rho^N}^p}{\log(n)^p} \mathbf{1}_{S_n \ge \frac{m}{2M}\left(1+\height{\pts_\rho^N}\right)-1} \right]
        \\&\leq \expec\left[ \height{\pts_\rho^N}^p e^{-\frac{m^2}{2M^2}\left(1+\height{\pts_\rho^N}\right)} \right]
        + \expec\left[ \frac{\big((2M/m)\cdot(S_n+1)-1\big)^p}{\log(n)^p} \right] \,.
    \end{align*}
    Since the function $x\mapsto x^pe^{-\frac{m^2}{2M^2}(1+x)}$ is bounded over $\bR_+$, the first term is bounded in $n$.
    As for the second term, we use \eqref{eq:domin top bst binom} along with $(a+b)^p\leq2^{p-1}(a^p+b^p)$ to deduce:
    \begin{align*}
        \expec\left[ \frac{\big((2M/m)\cdot(S_n+1)-1\big)^p}{\log(n)^p} \right]&\le \left(\frac{2M}{m}\right)^p\expec\left[ \frac{(S_n+1)^p}{\log(n)^p} \right]
        \\&\le 2^{2p-1}\left(\frac{M}{m}\right)^p\left(\expec\left[ \frac{\height{\pts_-}^p}{\log(n)^p} \right]+\frac{1}{\log(n)^p}\right)
    \end{align*}
    which is bounded in $n$ by \cite[Lemma 3.1]{devroye1986note} and Poisson estimates.
    This concludes the proof of the uniform integrability claim.
\end{proof}

The weakness of the previous proposition is that $\eta$, which depends on the rectangle under consideration, might be big.
In the next statement we show that, for continuous positive densities $\rho$, it is possible to choose rectangles for which the corresponding $\eta$ is small.

\begin{cor}\label{cor:control-top-tree}
    Let $D$ be a compact domain in the plane and $\rho: D \to(0,\infty)$ be a continuous, positive intensity function.
    Then for any $\eps>0$, there exists $\beta>0$ such that for any $\beta'\le\beta$ and any rectangle $R=[x_1,x_1+\beta']\times[y_1,y_2]$ with non-empty interior contained in $D$:
    \begin{align*}
        \lim_{n\rightarrow\infty} \prob\left[ \left|\frac{\height{\pts_\rho^N\cap R}}{c^*\log n}-1\right| >\eps \right] =0 \,.
    \end{align*}
\end{cor}

In particular, taking $x_1=y_1=0$ and $y_2=1$, the tree $\lbst{\pts_\rho^N\cap R}$ is the top tree $\lbst{\pts_\rho^N(\beta)}$ defined at the beginning of \Cref{sec:height-proof}.

\begin{proof}
    Let $\eps>0$ and assume that $\eps< \min_D\rho$. 
    By uniform continuity of $\rho$, find $\beta>0$ such that for any $(x,y),(x',y')\in D$:
    \begin{equation*}
        |x-x'| + |y-y'| \le \beta \quad\Longrightarrow\quad |\rho(x,y)-\rho(x',y')| \le \eps .
    \end{equation*}
    Then fix $\beta'\le\beta$ and consider any rectangle $R=[x_1,x_1+\beta']\times[y_1,y_2]$ contained in $D$.
    Define 
    \[f : y \in[y_1,y_2] \mapsto \int_{y_1}^y \rho(x_1,t)dt 
    \quad\text{ and }\quad 
    g : y\in[y_1,y_2] \mapsto y_1 + (y_2-y_1)f(y)/f(y_2).\]
    The function $g$ is a $\mathcal{C}^1$ increasing map from $[y_1,y_2]$ onto itself.
    Let $\widetilde\pts$ denote the set of points obtained after applying the transformation 
    $(x,y) \mapsto \left(x,g(y) \right)$
    to $\pts^N_\rho\cap R$.
    This transformation does not change the relative orders of points, therefore $\lbst{\widetilde\pts}$ and $\lbst{\pts^N_\rho\cap R}$ {have the same shape}.
    Additionally, $\widetilde\pts$ follows the law of a Poisson point process with intensity 
    \[ n\, \frac{\rho(x,g^{-1}(y))}{g'(g^{-1}(y))} =  n\, \frac{f(y_2)}{y_2-y_1} \frac{\rho(x,g^{-1}(y))}{\rho(x_1,g^{-1}(y))} \] 
    on $R$.
    Thus we can apply \Cref{prop:control-top-tree} with 
    \[ m= \frac{f(y_2)}{y_2-y_1}
    \left(1-\frac{\eps}{\min_D\rho}\right)
    \,,\quad 
    M= \frac{f(y_2)}{y_2-y_1} \left(1+\frac{\eps}{\min_D\rho}\right)
    \,,\quad 
    \eta=\frac{2\eps}{\min_D\rho-\eps} \,, \]
    to obtain:
    \begin{align*}
        \lim_{n\rightarrow\infty}\prob\left[ \left|\frac{\height{\tilde\pts}}{c^*\log n}-1\right| >\eta+\eps \right] =0\,.
    \end{align*}
    Since this holds for any small enough $\eps>0$ and since $\eta$ goes to $0$ when $\eps$ goes to $0$, the result follows.
\end{proof}

\subsection{Some bounds on the hanging trees}

As above, let $\pts^N$ be a Poisson point process on $[0,1]^2$ with intensity $n\mu$.
We use the notations at the beginning of \Cref{sec:height-proof}.
Moreover, for each $k \le |\pts^N(\beta)|$, we let $\zeta_k=|I_k|$ be the size of the $k$-th gap, and $\mathcal P^N_k(\beta)$ be the points of $\pts^N$ in the horizontal band $(\beta,1] \times I_k$.
The sizes of the bands and the number of points in each band are then controlled by the following proposition.

\begin{prop}\label{prop:tight rows}
    Let $\mu$ be a permuton.
    Assume that there exists $\beta>0$ such that $\mu_{|[0,\beta]\times[0,1]}$ has a continuous and positive density $\rho:[0,\beta]\times[0,1]\rightarrow(0,\infty)$.
    Then the following holds.
    \begin{enumerate}
      \item \hypertarget{item1}{There exists $\alpha$ such that 
        \[ \lim_{n \to +\infty} \prob\left[ \max_k \, \zeta_k > \alpha\frac{\log n}{n} \right] =0.\]}
       \item \hypertarget{item2}{The sequence of random variables
        \begin{align*}
            \left(\frac{1}{\log n} \max_{0\leq k\leq|\pts(\beta)|}\big|{\pts_k^N(\beta)}\big| \right)^p
        \end{align*}
        is uniformly integrable for any $p \ge 1$.}
    \end{enumerate}
\end{prop}

\begin{proof}
    Let $\ell=\lceil n/ (a\log n) \rceil$, where the constant $a\ge1$ will be specified later, and let us divide vertically the rectangles $[0,\beta]\times[0,1]$ and $(\beta,1] \times [0,1]$, each into $\ell$ cells of the same size. 
    Namely we set, for $0\le i <\ell$,
    \[C_i=\big[0,\beta\big]\times \big[\tfrac{i}{\ell},\tfrac{i+1}{\ell}\big], 
         \quad D_i=\big(\beta,1\big] \times \big[\tfrac{i}{\ell},\tfrac{i+1}{\ell}\big].\]
    For each $i$, the probability that $\pts^N \cap C_i$ is empty equals $e^{-n \mu(C_i)} \le e^{-n \frac{\beta m}{\ell}}$, where $m$ is a lower bound for the density $\rho$ on $[0,\beta]\times[0,1]$.    
    Call $E$ the event that all $C_i$ contain at least one point of $\pts^N$.
    It follows from the above computation along with a union bound that
    \begin{align}\label{eq:prob_un_Ci_vide}
       \prob\big[ E^c \big]= \prob\Big[ \exists i:\, \pts^N \cap C_i = \emptyset \Big]
        &\leq \ell\, e^{-n \frac{\beta m}{\ell}}\leq \left\lceil \frac{n}{a\log n} \right\rceil e^{-\beta  m a \log n 
        {+ o(\log n)}
        }
    \end{align}
    as $n\to\infty$.
    When $E$ is satisfied, for all $k$ it holds that $\zeta_k \le \frac2{\ell} \le \frac{2a \log n}n$, implying
    \[\prob\left[ \max_k \, \zeta_k > \frac{2a \log n}n \right]
    \le \prob\big(E^c\big) \le \left\lceil \frac{n}{a\log n} \right\rceil e^{-\beta  m a \log n + o(\log n)}.\]
    The latter bound tends to $0$ if we choose $a >1/(\beta m)$, showing item~\itemlink{1}.
    
    For item~\itemlink{2} we observe that, assuming $E$, any given band $(\beta,1] \times I_k$ intersects at most two cells $D_i$, implying
    \[ \max_{0\leq k\leq|\pts(\beta)|} \big| \pts_k^N(\beta) \big| \le 2 \max_{0 \leq i \leq \ell-1} \big| \pts^N \cap D_i \big|.\]
    The random variables $\big| \pts^N \cap D_i \big|$ are $\poisson{n \mu(D_i)}$ distributed.
    Thanks to the bounds $\mu(D_i) \le \mu( [0,1] \times [\tfrac{i}{\ell},\tfrac{i+1}{\ell}])=\tfrac1{\ell} \le \tfrac{a\log n}{n}$, this implies the following estimates for $b>0$:
    \begin{align*}
        \prob\left[ \max_{0\leq k\leq|\pts(\beta)|} \big| \pts_k^N(\beta) \big| \ge 2 b \log n \,,\, E \right]
        &\le \prob\left[ \max_{0 \leq i \leq \ell-1} \big| \pts^N \cap D_i \big| \ge  b \log n \right]
        \\& \le \ell \, \prob\Big[ \poisson{a \log n} \ge b \log n \Big] .
    \end{align*}
    Using \Cref{lem:Poisson_bounds}, for $b > a$ we deduce
    \begin{equation}\label{eq:bound_max_geq_2blogn}
        \prob\left[ \max_{0\leq k\leq|\pts(\beta)|} \big| \pts_k^N(\beta) \big| \ge 2b \log n \,,\, E \right]
        \le \left( \frac{e a}{b} \right)^{b \log n} \left\lceil \frac{n}{a\log n}\right\rceil e^{-a \log n}. 
    \end{equation}
    For $a \geq 1$ and $b \ge e a$ the right-hand side is uniformly bounded by $(ea/b)^b$ for any $n\geq e$.

    We now have all the necessary estimates to conclude the proof of item~\itemlink{2}.
    Recall that we want to prove uniform integrability of all powers of $\frac{1}{\log n} \max_{0\leq k\leq|\pts(\beta)|} | \pts_k^N(\beta) |$, which is equivalent to boundedness in $n$ of all its moments.
    Fix $p > 0$. 
    We have, using the layer cake representation and~\eqref{eq:prob_un_Ci_vide} and~\eqref{eq:moment_poisson}:
    \begin{multline}
      \label{eq:bound_pth_moment_max}
      \expec\left[\left( \frac{\max_{0\leq k\leq|\pts(\beta)|} \big| \pts_k^N(\beta)\big|}{\log n} \right)^p \right]
     \le \expec\left[ \big|\pts^N\big|^p \Big] \prob\big[ E^c \right]
     + \expec\left[\left( \frac{\max_k \big| \pts_k^N(\beta)\big|}{\log n}  \right)^p \One_E \right] \\
     \le n^p (1+o(1)) 
     \left\lceil \frac{n}{a\log n} \right\rceil e^{-\beta  m a \log n
     {+o(\log n)}
     } + p \int_0^{\infty} s^{p-1}
    \prob\left[ \frac{\max_k \big| \pts_k^N(\beta)\big|}{\log n} \ge s ,  E \right] ds .
    \end{multline}
    We choose $a>\frac{p+1}{\beta m}$, so that the first term tends to $0$, and thus, is a bounded sequence in $n$.
    The second term is bounded as follows, using \eqref{eq:bound_max_geq_2blogn} for $s \ge 2ea$ and simply bounding the probability by $1$ otherwise:
    \[\int_{0}^{\infty} s^{p-1}
    \prob\left[ \frac{\max_{0\leq k\leq|\pts(\beta)|} \big| \pts_k^N(\beta) \big|}{\log n} \ge s ,  E \right] ds
    \le \int_{0}^{\infty} s^{p-1} \min\left(1,\left(\frac{2ea}{s} \right)^{s/2}\right) ds \,.
     \]
     This integral is finite and independent of $n$, and we conclude that \eqref{eq:bound_pth_moment_max} is bounded in $n$. 
     The proposition is proved.
\end{proof}

Item~\itemlink{1} can be further used to control the maximal height of a hanging tree in $\bst{\pts^N_\mu}$.

\begin{prop}\label{prop:height_hanging_trees}
  Let $\mu$ be a permuton satisfying \assumption{1}, i.e.~$\mu$ has an upper bounded density $\rho$ on $[0,1]$, which is positive and continuous on $[0,\beta]\times[0,1]$ for some $\beta>0$. 
  Then we have the following convergence in probability as $n$ goes to infinity:
  \[ \frac1{\log n} \max_{0\leq k\leq|\pts(\beta)|}\Big\{\height{\pts^N_k(\beta)}\Big\} \longrightarrow 0.\]
\end{prop}

\begin{proof}
From \Cref{prop:tight rows}, item~\itemlink{1}, there exists $\alpha>0$ such that $\max_k \zeta_k < \alpha\frac{\log n}{n}$ w.h.p.~as $n\to\infty$.
We work conditionally given $\pts^N(\beta)$, and assume that $\max_k \zeta_k < \alpha\frac{\log n}{n}$.
In particular, as before, we let $\{y_{(1)}<\dots<y_{(K_\beta)}\}$ be the ordered $y$-coordinates of the points in $\pts(\beta)$, with $K_\beta=|\pts(\beta)|$, and set by convention $y_{(0)}=0$ and $y_{(K_\beta+1)}=1$.
Then for each $0\le k\le |\pts(\beta)|$, the family $\pts_k^N(\beta)$ is distributed like a Poisson point process with intensity $n \rho_{|[\beta,1]\times[y_{(k)},y_{(k+1)}]}$.

For any $0\le k \le |\pts(\beta)|$, since $\zeta_k < \alpha\frac{\log n}{n}$, \Cref{cor: extreme deviation BST height} applies with $\rho$ restricted to $[\beta,1]\times[y_{(k)},y_{(k+1)}]$ and $\zeta=\alpha\frac{\log n}{n}$.
Thus for $n \zeta=\alpha \log n$ large enough:
\begin{equation*}              
    \prob\left[ \height{\pts_k^N(\beta)} > \eps \log n \right]
    = \prob\left[ \height{\pts_k^N(\beta)} > 2 \frac{\eps}{2\alpha} \zeta n \right] \le 4 \exp \left[-\frac{\eps}{4\alpha} (\alpha \log n) \log(\alpha \log n) \right].
\end{equation*}
A union bound then implies that, still conditionally given the family $\pts(\beta)$ and assuming that $\max_k \zeta_k < \alpha\frac{\log n}{n}$, one has
  \[\prob\left[ \frac1{\log  n} \max_k\height{\pts_k^N(\beta)} > \eps \right]
  \le (|\pts(\beta)|+1) \,\cdot\, 4 \exp \left[-\frac{\eps}{4\alpha} (\alpha \log n) \log(\alpha \log n) \right].\]
But w.h.p., the inequality $\max_k \zeta_k < \alpha\frac{\log n}{n}$ indeed holds and $|\pts(\beta)| <n$, so the unconditioned probability tends to $0$ as $n$ tends to infinity:
\[\prob\left[ \frac1{\log  n} \max_{0\leq k\leq|\pts(\beta)|}\height{\pts_k^N(\beta)} > \eps \right]
  \underset{n\rightarrow\infty}{\longrightarrow}0.\]
This holds for any $\eps>0$, proving the proposition.
\end{proof}

\begin{remark}
    Item (1) in Proposition~\ref{prop:tight rows} could alternatively have been derived using standard results on the maximal gap, also called {\em maximal spacing}, between i.i.d.~uniform random variables; see e.g.~\cite{slud1978gap}.
    Indeed, by applying a thinning procedure, $\max_k\zeta_k$ is bounded above by the maximal gap between $\poisson{n\beta m}$ i.i.d.~uniform variables in $[0,1]$, which is known to concentrate around $\log(n)/(n\beta m)$.
\end{remark}

\subsection{Concluding the proof of the height theorem}

First, we can deduce uniform integrability of all powers of $\frac{\height{\pts^N}}{\log n}$, under a hypothesis which is slightly weaker than \assumption{1} (more precisely, we only
make an assumption regarding the behavior on $[0,\beta]\times[0,1]$).

\begin{prop}\label{prop:UI}
    Let $\mu$ be a permuton.
    Assume that there exists $\beta>0$ such that $\mu_{|[0,\beta]\times[0,1]}$ has a continuous and positive density $\rho:[0,\beta]\times[0,1]\rightarrow(0,\infty)$.
    Then, for any $p>0$, the family of random variables
    \begin{align*}
        \left(\left(\frac{\height{\pts^N_\mu}}{\log n}\right)^p\right)_{n\geq2}
    \end{align*}
    is uniformly integrable.
\end{prop}

\begin{proof}
    This follows immediately from \Cref{lem:height bound}, together with Propositions~\ref{prop:control-top-tree}
    and~\ref{prop:tight rows},
    using the trivial bound $\height{\pts_k^N(\beta)} \le \big|{\pts_k^N(\beta)}\big|$
    for the hanging trees.
\end{proof}

Now we can combine our results to establish \Cref{thm:height} under Assumption~\assumption{1}.

\begin{proof}[Proof of \Cref{thm:height}]
Thanks to \Cref{thm:Poisson fixed}, it suffices to prove the theorem in its Poissonized version.
We shall work with $\pts^N$, a Poisson point process of intensity $n\mu$.

Fix $\eps>0$. 
Let $D$ be a compact neighborhood of $\{0\}\times[0,1]$ on which $\rho$ is continuous and positive, and let $\beta=\beta(\eps)>0$ be given by \Cref{cor:control-top-tree} applied to $\rho$ on $D$.
Therefore
\begin{align*}
    \lim_{n\rightarrow\infty} \prob\left[ \left|\frac{\height{\pts^N(\beta)}}{c^*\log n}-1\right| >\eps \right] =0
\end{align*}
where $\lbst{\pts^N(\beta)} = \lbst{\pts^N\cap ([0,\beta]\times[0,1])}$ is the top tree of $\lbst{\pts^N}$.
Furthermore, by \Cref{prop:height_hanging_trees}, the quantity
\[\frac1{\log n} \max_{0\leq k\leq|\pts(\beta)|}\Big\{\height{\pts_k(\beta)} \Big\}\]
converges in probability to $0$.
Combining this with \Cref{lem:height bound}, we conclude that the random variables $\frac1{\log n} \height{\pts^N}$ converge in probability to $c^*$.
Finally, \Cref{prop:UI} implies uniform integrability of all powers, and thus $L^p$ convergence for all $p\ge1$.
\end{proof}

\section{Examples and extra results}\label{sec:examples and extras}

\subsection{The Mallows permuton}\label{ssec:example_Mallows}

Fix $\gamma\in\bR$, and let $\nu_\gamma$ be the permuton with density
\begin{equation*}
    \rho_\gamma(x,y) := 
    \frac{\gamma \sinh(\gamma)}{\left( e^{\gamma/2}\cosh\big( \gamma(x-y) \big) - e^{-\gamma/2}\cosh\big( \gamma(x+y-1) \big) \right)^2} 
\end{equation*}
for $(x,y)\in[0,1]^2$.
The permuton $\nu_\gamma$ appears as the limit of Mallows random permutations, as introduced in~\cite{mallows1957non}.
We recall that $\sigma_{n,q}$ is a Mallow random permutation of size $n$ and parameter $q$ if for $\tau \in S_n$, the probability $\mathbb P(\sigma_{n,q}=\tau)$ is proportional to $q^{i(\tau)}$,
where $i(\tau)$ is the number of inversions in $\tau$.
When $q=q_n=1-2\gamma\, n^{-1}+o(n^{-1})$, it has been proved in \cite{starr2009permuton-Mallows} that $\sigma_{n,q}$ (or more precisely its associated permuton) converges to the permuton $\nu_\gamma$.

The permuton $\nu_\gamma$ satisfies Assumption~\assumption{1}, therefore $\lbst{\sigma^n_{\nu_\gamma}}$ has height $(c^*+o(1)) \log n$, by \Cref{thm:height}.
This is not surprising, since it was proved in \cite{addario2021bst-mallows} that, in the regime $q=q_n=1-2\gamma\, n^{-1}+o(n^{-1})$, the tree $\lbst{\sigma_{n,q}}$ has height $(c^*+o(1)) \log n$ as well.
The asymptotics of $\height{\sigma^n_{\nu_\gamma}}$ can not be directly deduced from that of $\height{\sigma_{n,q}}$ or vice versa since $\sigma_{n,q}$ and $\sigma^n_{\nu_\gamma}$ have different distributions, but these two random permutations can be coupled in a rather strong way (see \cite{mueller2013Mallows-LIS}, where such a coupling is constructed to study the longest increasing subsequence), and it would have been surprising that the heights of their BSTs behave differently.

Since Assumption~\assumption{2} is weaker that \assumption{1}, we can also apply \Cref{thm:limit} to the permuton $\nu_\gamma$, where the derivative $\nu_{\gamma,0}$ has density
\begin{equation*}
    \rho_{\gamma,0}(y) = \rho_\gamma(0,y) =
    \frac{\gamma \sinh(\gamma)}{\left( e^{\gamma/2}\cosh( \gamma y ) - e^{-\gamma/2}\cosh\big( \gamma(y-1) \big) \right)^2} 
\end{equation*}
for $y\in[0,1]$.
In particular the measure $\nu_{\gamma,0}$ is \textit{not} the uniform measure on $[0,1]$, therefore the random function $\psi_{\nu_{\gamma,0}}$ is \textit{not} distributed like $\psi_{\Leb_{[0,1]}}$.
In other words: $\lbst{\sigma^n_{\nu_\gamma}}$ has a different (random) subtree size limit than the BST of a uniform permutation $\lbst{\sigma^n_{\Leb_{[0,1]^2}}}$.

\subsection{Permutons with partially vanishing densities on the left edge and binary search trees of polynomial height}
\label{sec:hypothesis-A1}

Let $E := \big\{(x,y) \in [0,1]^2: x \le y \le x+\tfrac12 \text{ or }y\le x-\tfrac12\big\}$, see \Cref{fig:Sets_E_Ein} (left).
It is straight-forward to check that the measure $\mu$ defined by $\mu(A) := 2\Leb(A \cap E)$ has uniform marginals, i.e.~is a permuton.

\begin{figure}[htb]
    \centering
    \includegraphics{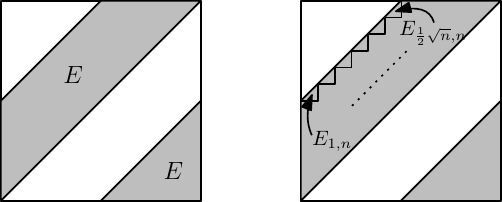}
    \caption{Left: the support $E$ of the permuton of \Cref{prop:deep_BST}. Right: the sets $E_{i,n}$ involved in the proof.}
    \label{fig:Sets_E_Ein}
\end{figure}

\begin{prop}
\label{prop:deep_BST}
    Let $\mu$ be the above permuton. Then, for any $\eps>0$:
    \[\lim_{n\to\infty} \prob\left[\Big. \height{\pts_\mu^n} \ge \tfrac12(1-\tfrac1e - \eps) \sqrt n \right] = 1. \]
\end{prop}

\begin{proof}
We first consider a Poisson point process $\pts_\mu^N$ of intensity $n\mu$.
We say that a point $(x,y)$ in $\pts_\mu^N$ is a record if there is no point in $([0,x) \times (y,1]) \cap \pts_\mu^N$, i.e.~if there is no point of $\pts_\mu^N$ above and to the left of $(x,y)$.
It is easily seen that in the construction of $\lbst{\pts_\mu^N}$, a point is inserted in the right-most branch of the tree if and only if it is a record.
Hence the number of nodes on that right-most branch is the number of records in $\pts_\mu^N$, which we will denote by $\rec(\pts_\mu^N)$. 
Therefore, $\height{\pts_\mu^N} \ge \rec(\pts_\mu^N)-1$, and we will prove a lower bound in probability for $\rec(\pts_\mu^N)$.

For any positive integer $i \le \tfrac12 \sqrt n$, let us define $x_i := i/\sqrt n$ and
\[ E_{i,n} :=\{ (x,y): x_{i-1} \le y-\tfrac12 \le x \le x_i \} \subset E \,\]
as shown on \Cref{fig:Sets_E_Ein}. 
We have $\mu(E_{i,n})=2\Leb(E_{i,n})=1/n$, implying that each $E_{i,n}$ contains $\poisson{1}$ points in $\pts^N_\mu$.
Moreover, these numbers are independent random variables.
Hence, for any $\eps>0$, w.h.p., at least a fraction $(1-\tfrac1e - \eps)$ of the sets $E_{i,n}$ (for $i \le \tfrac12 \sqrt n$) contain a point in $\pts^N_\mu$.
Each non-empty $E_{i,n}$ contains at least one record, implying $\rec(\pts_\mu^N) \ge \tfrac12(1-\tfrac1e - \eps) \sqrt n$ w.h.p.
This shows
\[\lim_{n\to\infty} \prob\left[\Big. \height{\pts_\mu^N} \ge \tfrac12(1-\tfrac1e - \eps) \sqrt n \right] = 1. \]
The same result for $\pts_\mu^n$ is deduced using the de-Poissonization techniques of \Cref{thm:Poisson fixed}.
\end{proof}

On the other hand, for any permuton $\mu$ with a bounded density $\rho$, w.h.p.~it holds that $\height{\pts_\mu^n} = \O(\sqrt{n})$.
Indeed, this follows from \Cref{lem: upper bound LIS LDS} and the {proof method of \cite[Proposition~1.3]{dubach2023LIS-inermediate}}.
Hence the lower bound given in~\Cref{prop:deep_BST} is optimal up to a multiplicative constant.

The above example can be modified to a permuton with a continuous density.
Take a continuous function $\varphi:[0,1/4] \to \mathbb R_+$ such that $\varphi(0)=0$ and $\int_0^{1/4}\varphi(t) dt=1/2$.
We then define \[ \rho(x,y)=\varphi\left( \dist_{L_1}\big((x,y),E^c \cup \{(1,0)\}\big) \right) \,\]
Going cyclically along any horizontal line, the $L_1$ distance above grows linearly from $0$ to $1/4$ and then decreases linearly again from $1/4$ to $0$. 
The same holds along vertical lines.
Since $\int_0^{1/4}\varphi(t) dt=1/2$, this implies that the measure $\mu=\rho(x,y)dx dy$ is a permuton supported on $E$.
Moreover, with the notation of the above proof, we have 
\[\mu(E_{i,n})=\int_0^{1/\sqrt n} (1/\sqrt n-t) \varphi(t) dt.\]
Choosing e.g.~$\varphi(t)\sim t^\delta$ for small $t$, we have 
\[\mu(E_{i,n}) \sim \frac{n^{-1-\delta/2}}{(\delta+1)(\delta+2)},\]
implying that each $E_{i,n}$ contains a point in $\pts^N_\mu$ with probability roughly $n^{-\delta/2}$.
Hence, w.h.p.~at least $\Theta(n^{(1-\delta)/2})$ sets among the $E_{i,n}$ are non-empty, showing that the height $\height{\pts_\mu^N}$ has polynomial growth.
This illustrates the importance of the positivity assumption on the density near the left edge of the square made in \Cref{thm:height}.
\medskip

Regarding the subtree size convergence of the permuton from \Cref{prop:deep_BST}, we remark that it satisfies Assumption~\assumption{2} with $\mu_0 = \Leb_{[0,1/2]}$.
Hence, the associated BSTs admit a subtree size limit which is different from the uniform case: 
looking at any fixed depth, asymptotically, \Cref{thm:limit} states that more than half of the nodes of $\lbst{\pts_\mu^n}$ belong to the right-most branch, as is expected from the shape of the permuton observed in \Cref{fig:Sets_E_Ein}.

\subsection{Permutons with positive densities on a band and binary search trees of large logarithmic height}
\label{ssec:logarithmic height but greater constant}

For $\beta >0$, we consider the permuton $\mu_{\beta}$ which has a mass $\beta$ uniformly distributed on the band $[0,\beta]\times[0,1]$ and a mass $1-\beta$ uniformly distributed on the line segment from $(\beta,0)$ to $(1,1)$.
These permutons satisfy the following.

\begin{prop}
For any $\beta>0$ and $\eps>0$, we have
\begin{equation}
\label{eq:height-large-log}
     \lim_{n \to +\infty} \prob\left[ \height{\pts_{\mu_\beta}^n}
       \ge \frac{1-\beta}{\beta+\eps} \log n \right] =1.
\end{equation}
\end{prop}

\begin{proof}
As usual we first consider a Poissonized version, namely let $\pts_{\mu_\beta}^N$ be a Poisson point process of intensity $n \mu_\beta$.
As in \Cref{sec:height-proof}, we let $y_{(1)}<\dots<y_{(K_\beta)}$ be the ordered $y$-coordinates of the points in $\pts(\beta) $.
The number $K_\beta$ follows a $\poisson{\beta n}$ law and thus $K_\beta/(\beta n)$ converges to $1$ in probability.
Finally, we set $\zeta_*=\max_k y_{(k+1)}-y_{(k)}$ with the convention $y_{(0)}=0$ and $y_{(K_\beta+1)}=1$. Conditionally given $K_\beta$, this is the maximal gap (or maximal spacing) defined by $K_\beta$ i.i.d.~uniform random variables in $[0,1]$, and, from a result of \cite{slud1978gap}, the quotient $\frac{\zeta_* K_\beta}{\log(K_\beta)}$ converges to $1$ in probability.
Hence $\frac{\zeta_* \beta n}{\log(n)}$  converges to $1$ in probability.
 
We now observe that, since the points of $\pts_{\mu_\beta}^N$ with $x$-coordinates bigger than $\beta$ are in increasing order, all hanging trees $\lbst{\pts_k(\beta)}$ consist of a single right branch, and thus $\height{\pts_k(\beta)}=|\pts_k(\beta)|-1$ for any $k \le K_\beta$.
Conditionally given $\pts(\beta)$, and letting $k_0$ be such that $y_{(k_0+1)}-y_{(k_0)}=\zeta_*$, the number of elements in $\pts_{k_0}(\beta)$ concentrates around $n(1-\beta) \zeta_*$, which is close to $\frac{1-\beta}{\beta} \log n$ in probability.
Since $\height{\pts_{\mu_\beta}^N} \ge \height{\pts_{k_0}(\beta)}$, this proved \eqref{eq:height-large-log} with $\pts_{\mu_\beta}^n$ replaced by its Poissonized version $\pts_{\mu_\beta}^N$.
The proposition follows using the de-Poissonization techniques of \Cref{thm:Poisson fixed}.
\end{proof}

Our bound is once again optimal up to a multiplicative constant.
Indeed, the permutons $\mu_\beta$ satisfy the hypotheses of \Cref{prop:UI}, which proves that all powers of $\frac{\height{\pts_{\mu_\beta}^n}}{\log n}$ are uniformly integrable.
\medskip

Moreover, the permuton $\mu_\beta$ clearly satisfies Assumption~\assumption{2}, with left-derivative $\mu_{\beta,0} = \Leb_{[0,1]}$.
By \Cref{thm:limit}, the BSTs $\lbst{\pts_{\mu_\beta}^n}$ therefore have the same subtree size limit as in the uniform case.

\subsection{Permutons with no subtree size limit}\label{ssec:no_STS}

In this section, we exhibit two permutons for which the BSTs have no subtree size limit.
The first one, $\mu^{(1)}$, does not satisfy \eqref{eq:def_mu0}, therefore $\lbst{\pts_{\mu^{(1)}}^n}$ does not converge for the subtree size topology by \Cref{thm:limit}.
The second one, $\mu^{(2)}$, satisfies \eqref{eq:def_mu0} with $\mu_0 = \delta_{1/2}$, but we can show that $\lbst{\pts_{\mu^{(2)}}^n}$ does not converge for the subtree size topology.
Both are illustrated in \Cref{fig:permutons_no_STS}.

\begin{figure}
    \centering
    \includegraphics[scale=.9]{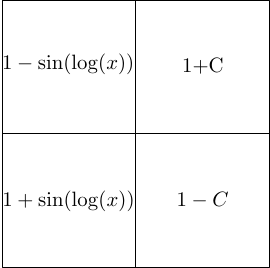}
    \qquad\qquad\qquad
    \includegraphics[scale=.9]{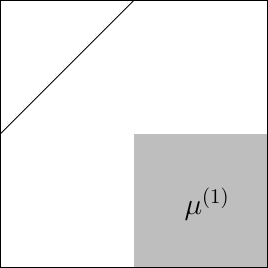}
    \caption{The permutons $\mu^{(1)}$ (on the left) and $\mu^{(2)}$ (on the right) constructed in \Cref{ssec:no_STS}.}
    \label{fig:permutons_no_STS}
\end{figure}

First, define $\mu^{(1)}$ as the permuton with density
\begin{equation*}
    f(x,y) =
    \left\{
    \begin{array}{llll}
    1 + \sin\left(\log x\right)
    & \text{if } (x,y)\in (0,1/2)\times(0,1/2) \\
    1 - \sin\left(\log x\right)
    & \text{if } (x,y)\in (0,1/2)\times(1/2,1) \\
    1-C
    & \text{if } (x,y)\in (1/2,1)\times(0,1/2) \\
    1+C
    & \text{if } (x,y)\in (1/2,1)\times(1/2,1) \\
    \end{array}
    \right.
\end{equation*}
where $C := \int_0^{1/2} \sin\left(\log x\right) dx$.
It is straightforward to check $\int_0^1 \int_0^t f(x,y)dxdy = \int_0^t \int_0^1 f(x,y)dxdy = t$ for all $t\in[0,1]$, therefore $\mu^{(1)}$ is indeed a permuton.

\begin{prop}\label{prop:permuton_not_2}
    The permuton $\mu^{(1)}$ does not satisfy \eqref{eq:def_mu0}, i.e.~the measures
    \begin{equation*}
        \frac1x \mu^{(1)}\left(\big. [0,x]\times\cdot \right)
    \end{equation*}
    do not converge as $x\to0^+$.
\end{prop}

\begin{proof}
    For any $(x,y)\in(0,1/2)^2$, we have:
    \begin{equation*}
        \frac1x \mu^{(1)}\left(\big. [0,x]\times[0,y] \right)
        = \frac1x \int_0^y \int_0^x f(s,t) dsdt
        = \frac{y}{x} \int_0^x 1 + \sin\left(\log s\right) ds
        = y + \frac{y}{2}\left(\big. \sin(\log x) - \cos(\log x)\right) .
    \end{equation*}
    For any fixed $y\in(0,1/2)$, this does not converge as $x\to0^+$.
    Therefore the distribution function of $\frac1x \mu^{(1)}\left(\big. [0,x]\times\cdot \right)$ does not converge at any fixed $y\in(0,1/2)$ as $x\to0^+$, and by the Portmanteau theorem, this concludes the proof.
\end{proof}

Now, we can use the permuton $\mu^{(1)}$ to construct a permuton $\mu^{(2)}$ which satisfies \eqref{eq:def_mu0}, but for which $\lbst{\pts_{\mu^{(2)}}^n}$ does not converge for the subtree size topology.
Informally, $\mu^{(2)}$ puts weight $1/2$ on straight line from $(0,1/2)$ to $(1/2,1)$ and contains a rescaled copy of $\mu^{(1)}$ of total mass $1/2$ in the lower right corner, see the right-hand side of \Cref{fig:permutons_no_STS}.
Formally, $\mu^{(2)}$ is characterized by:
\begin{equation*}
    \int_0^1 \int_0^1 h(x,y) \mu^{(2)}(dx,dy)
    = \int_0^{1/2} h(t, 1/2+t) dt
    + \int_0^1 \int_0^1 \frac12 h( (x+1)/2 , y/2 ) \mu^{(1)}(dx,dy)
\end{equation*}
for any bounded, measurable $h:[0,1]^2 \to \bR$.
Clearly, it satisfies \eqref{eq:def_mu0} with $\mu_0 = \delta_{1/2}$.

\begin{prop}
    The sequence $\lbst{\pts_{\mu^{(2)}}^n}$ does not converge for the subtree size topology as $n\to\infty$.
\end{prop}

\begin{proof}
    Let $\pts_{\mu^{(2)}}^N$ be a Poisson point process with intensity $n\mu^{(2)}$.
    Then 
    \[ \pts_{\mu^{(2)}}^N = \left( \pts_{\mu^{(2)}}^N \cap [0,1/2]\times[1/2,1] \right) \cup \left( \pts_{\mu^{(2)}}^N \cap [1/2,1]\times[0,1/2] \right) \]
    where $\pts_{\mu^{(2)}}^N \cap [0,1/2]\times[1/2,1]$ is an up-right sequence of points, and $\pts_{\mu^{(2)}}^N \cap [1/2,1]\times[0,1/2]$ is an affine transformation of a Poisson point process with intensity $\frac12 n \mu^{(1)}$.
    Therefore the subtree of $\lbst{\pts_{\mu^{(2)}}^N}$ rooted at the left child of the root is distributed like $\lbst{\pts_{\mu^{(1)}}^{N'}}$, where $N' \sim \poisson{n/2}$.
    By \Cref{prop:permuton_not_2} and \Cref{thm:limit}, the sequence $\lbst{\pts_{\mu^{(1)}}^{n}}$ does not converge for the subtree size topology, and the same holds for $\lbst{\pts_{\mu^{(1)}}^{N'}}$ by Poissonization.
    Hence $\lbst{\pts_{\mu^{(2)}}^N}$ does not converge for the subtree size topology, and the same holds for $\lbst{\pts_{\mu^{(2)}}^n}$ by dePoissonization.
\end{proof}

\subsection{A lower bound result}

As a last result in this paper, we emphasize that the lower bound in \Cref{thm:height} holds under a rather weak hypothesis.
This can be seen as a partial result towards \Cref{conj:uniform-optimal}.

\begin{prop}\label{prop:uniform_optimal_partial}
Let $\mu$ be a permuton. 
Suppose there exists $0\leq y\leq 1$ such that $\mu$ admits a continuous, positive density $\rho$ on a neighborhood of the point $(0,y)$.
Then for any $\eps>0$:
\begin{equation*}
    \lim_{n\to\infty} \prob\left[ \frac{\height{\pts_\mu^N}}{c^*\log n} \le 1-\eps \right] = 0.
\end{equation*}
\end{prop}

\begin{proof}
    This is relatively easy to prove, based on some intermediate results, established while proving \Cref{thm:height}.
    Let $D$ be a compact neighborhood of the point $(0,y)$ on which $\mu$ admits a continuous, positive density $\rho$.
    Fix $\eps>0$ and let $\beta=\beta(\eps)>0$ be given by \Cref{cor:control-top-tree}.
    We can take $0<\delta\le\beta$ so that there exists at least one rectangle $R_0=[0,\delta]\times[y_0-\delta,y_0+\delta]$ contained in $D$.
    Then by \Cref{cor:control-top-tree}: 
    \begin{equation*}
        \lim_{n\to\infty} \prob\left[ \frac{\height{\pts_\mu^N\cap R_0}}{c^*\log n} \le 1-\eps \right] = 0.
    \end{equation*}
    Moreover, using \Cref{lem:ancestor condition}, one can see that each chain of $\lbst{\pts_\mu^N\cap R_0}$ is still a chain in $\lbst{\pts^N_\mu}$.
    Therefore $\height{\pts_\mu^N\cap R_0} \le \height{\pts^N_\mu}$ a.s., and this concludes the proof.
\end{proof}

\section*{Acknowledgments}
The authors are grateful to Mathilde Bouvel for pointing out the work of \cite{grubel2023note} and for several stimulating discussions on the topic.
The authors are also grateful to the anonymous referee for their valuable comments, in particular for suggesting the converse statement
in \Cref{thm:limit}.
VF is partially supported by the Future Leader Program of the LUE (Lorraine Universit\'e d'Excellence) initiative.
Funds from this program were used for a visit of BC in Nancy, during which this project was initiated.
BC has received funding from the European Union's Horizon 2020 Research and Innovation Programme under the Marie Sk\l{}odowska-Curie Grant Agreement No.~101034253.

\bibliographystyle{bibli_perso} 
\bibliography{main}

\end{document}